\documentclass[12pt,a4paper,leqno]{amsart}

\usepackage[applemac]{inputenc}
\usepackage[T1]{fontenc} 
\usepackage{amsthm}
\usepackage{amsfonts}       
\usepackage{amsmath}
\usepackage{amssymb}
\usepackage{mathrsfs}

\newcommand{\R}{\mathbb{R}}

\newcommand{\Q}{\mathbb{Q}}

\newcommand{\Z}{\mathbb{Z}}
\newcommand{\ud}{\mathrm{d}}

\newcommand{\F}{\mathscr{F}}

\newcommand{\D}{\mathscr{D}}

\newcommand{\si}{\sigma}
\newcommand{\E}{\mathbb{E}}
\newcommand{\de}{\Delta}

\newcommand{\s}{\mathscr{S}}

\newcommand{\G}{\mathscr{G}}

\theoremstyle{plain}
\newtheorem{theorem}{Theorem}
\newtheorem{lem}[theorem]{Lemma}
\newtheorem{cor}[theorem]{Corollary}

\theoremstyle{definition}

\theoremstyle{remark}

\numberwithin{equation}{section}
\numberwithin{theorem}{section}

\title[Two weight $L^{p}$-inequalities]{Two weight $L^{p}$-inequalities for dyadic shifts and the dyadic square function}

\author{Emil Vuorinen}
\address{DEPARTMENT OF MATHEMATICS AND STATISTICS, P.O.B 68 (GUSTAF H\"ALLSTR\"OMIN KATU 2B), FI-00014 UNIVERSITY OF HELSINKI, FINLAND}
\email{emil.vuorinen@helsinki.fi}


\begin{document}

\begin{abstract}
We consider two weight $L^{p}\to L^{q}$-inequalities for dyadic shifts and the dyadic square function with general exponents $1<p,q<\infty$. It is shown that if a so-called quadratic $\mathscr{A}_{p,q}$-condition related to the measures holds, then a family of dyadic shifts satisfies the two weight estimate in an $\mathcal{R}$-bounded sense if and only if it satisfies the direct- and the dual quadratic testing condition. In the case $p=q=2$ this reduces to the result by T. Hyt\"onen, C. P\'erez, S. Treil and A. Volberg \cite{HPTV}. 

The dyadic square function satisfies the two weight estimate if and only if it satisfies the quadratic testing condition and the quadratic $\mathscr{A}_{p,q}$-condition holds. Again in the case $p=q=2$ we recover the result by F. Nazarov, S. Treil and A. Volberg \cite{NTV2}.

An example shows that in general the quadratic $\mathscr{A}_{p,q}$-condition is stronger than the Muckenhoupt type $A_{p,q}$-condition.
\end{abstract}

\maketitle

\section{Introduction}
The main purpose of this note is to consider \emph{two weight norm inequalities} for dyadic shifts and the dyadic square function. A two weight $L^{p} \to L^q$-inequality, $1<p,q<\infty,$ for an operator $T$ defined for a suitable class of functions would mean an inequality of the form
\begin{equation}\label{two weight}
\Big( \int_{\R^{n}} |Tf|^{q} w \ud x \Big)^{\frac{1}{q}} \leq C \Big( \int_{\R^{n}} |f|^{p} v \ud x \Big) ^{\frac{1}{p}},
\end{equation}
where the constant $C>0$ does not depend on $f$. Here $v$ and $w$ are \emph{weights}, that is, non-negative Borel measurable functions. The two weight inequality (\ref{two weight}) can also be formulated a little differently, and we will do so, but this type of a problem we are anyway working with. 

Dyadic shifts are in a sense discrete models of Calder\'on-Zygmund singular integral operators. They are much simpler than a general Calder\'on-Zygmund operator but they already have the complication that they are not \emph{positive} integral operators.

The sense in which we mean that the dyadic shifts represent Calder\'on-Zygmund operators is that it was shown in \cite{H} that a general Calder\'on-Zygmund operator can be represented as an average over all dyadic systems on $\R^{N}$ of a rapidly convergent series of dyadic shifts. This representation was used to prove the so called $A_{2}$-\emph{conjecture} about sharp constants in \emph{one weight} estimates for Calder\'on-Zygmund operators. 

Dyadic shifts fall also in the category of \emph{well localized operators}  as defined in \cite{NTV} by F. Nazarov, S. Treil and A. Volberg.  They showed that a two weight inequality holds for a well localized operator in $L^{2}$ if and only if the operator satisfies the so called \emph{Sawyer type testing conditions}. This means that it suffices to show that the operator and its formal adjoint satisfy the inequality with an arbitrary indicator of a (dyadic) cube, and hence the Sawyer type testing may also be called \emph{indicator testing.} Two weight $L^{p} \to L^q$-inequalities for well localized operators were considered in \cite{V}.

The definition of a well localized operator depends on a parameter $r$ which measures how ``well'' the operator is localized. The constant $C$ in the two weight inequality proved in \cite{NTV} and \cite{V} depends on $r$ and the constants in the Sawyer type testing conditions. 

In \cite{HPTV} the dyadic shifts were looked at from a little different perspective. There T. Hyt\"onen, C. P\'erez , S. Treil and A. Volberg proved the two weight inequality in $L^{2}$ assuming the Sawyer type testing conditions and finiteness of the so called $A_{2}$-constant related to the weights. This approach was related to the $A_{2}$-conjecture mentioned above, and this is the point of view that we take in this note. The main difference between this approach and the more general point of view of well localized operators is that this way one gets a better estimate depending on the \emph{complexity} of the shift, which was crucial in the $A_{2}$-conjecture. The complexity of the shift is somewhat analogous to the ``well localization'' parameter in the definition of well localized operators.

Our novelty here is that we characterize the two weight inequality for dyadic shifts for general exponents $1<p,q<\infty$, whereas it was only done before in the case $p=q=2$. Despite the positive result in the case $p=q=2$, F. Nazarov has constructed an example (unpublished) of a Haar multiplier (a special kind of dyadic shift) and a pair of weights such that the operator satisfies the Sawyer type testing conditions for some exponent $1<p=q<\infty$, $p\not=2$, but still does not satisfy the (quantitative) two weight estimate. See \cite{V}, Section 4, for a more precise statement of the example.

Knowing that there are problems with the Sawyer type testing and general exponents $p\in(1,\infty)$, we generalize the testing conditions for exponents $1<p<\infty$ in the spirit of $\mathcal{R}$-bounded operator families as used for example in \cite{W}. We call these new testing conditions \emph{quadratic testing conditions}. Similarly we interpret the $A_{2}$-condition as a special case of a \emph{quadratic $\mathscr{A}_{p,q}$-condition}, see Section \ref{square Ap condition} for a definition. 

Now we state a special version of the main Theorem \ref{two weight shift} for the dyadic shifts. It is assumed here that we have some fixed underlying dyadic lattice $\D$ on $\R^{N}$ which is used in the definition of the shifts and the $\mathscr{A}_{p,q}$-condition. 

\begin{theorem}
Fix exponents $p,q \in (1,\infty)$, and assume that  $\si$ and $w$ are two measures on $\R^N$ satisfying the quadratic $\mathscr{A}_{p,q}$-condition. Suppose $T^{\si}$ is a dyadic shift with complexity $\kappa$, and let $T^{w}$ be the formal adjoint of $T^{\si}$. Then there exists a constant $C$ such that
\begin{equation}\label{single shift}
\|T^{\si}f\|_{L^{q}(w)} \leq C \|f\|_{L^{p}(\si)}
\end{equation}
holds for all $f \in L^{p}(\si)$ if and only if there exist constants $C'$ and $C''$ such that for all sequences $(Q_{i})_{i=1}^{\infty} \subset \D$ of dyadic cubes and all sequences $(a_{i})_{i=1}^{\infty}$ of real numbers the inequalities
\begin{equation}\label{testing 1}
\Big\| \Big( \sum_{i=1}^{\infty} \big(a_{i} 1_{Q_{i}}T^{\si} 1_{Q_{i}}\big)^{2} \Big)^{\frac{1}{2}} \Big\|_{L^{q}(w)}
\leq C' \Big\| \Big( \sum_{i=1}^{\infty} a_{i}^{2}1_{Q_{i}} \Big)^{\frac{1}{2}} \Big\|_{L^{p}(\si)}
\end{equation}
and
\begin{equation}\label{testing 2}
\Big\| \Big( \sum_{i=1}^{\infty} \big(a_{i} 1_{Q_{i}}T^{w} 1_{Q_{i}}\big)^{2} \Big)^{\frac{1}{2}} \Big\|_{L^{p'}(\si)}
\leq C'' \Big\| \Big( \sum_{i=1}^{\infty} a_{i}^{2}1_{Q_{i}} \Big)^{\frac{1}{2}} \Big\|_{L^{q'}(w)}
\end{equation}
hold.

Moreover if $\mathcal{T}^{\si}$ and $\mathcal{T}^{w}$ denote the best possible constants in (\ref{testing 1}) and (\ref{testing 2}), respectively, and $[\si,w]_{p,q}$ is the quadratic $\mathscr{A}_{p,q}$-constant, then the best constant $\|T\|$ in (\ref{single shift}) satisfies
\begin{equation}\label{quantitative}
\|T\| \lesssim(1+\kappa)(\mathcal{T}^{\si} + \mathcal{T}^{w}) + (1+\kappa)^{2}[\si,w]_{p,q}.
\end{equation}
\end{theorem}

If $p=q=2$, quadratic testing is equivalent with indicator testing and the quadratic $\mathscr{A}_{2,2}$-condition is equivalent with the simple $A_{2}$-condition. Thus, when $p=q=2$, the above theorem reduces to the one proved in \cite{HPTV}. 

As an other novelty in Theorem \ref{two weight shift} we shall actually consider a family $\mathscr{T}$ of dyadic shifts with at most a given complexity $\kappa$. Then it is shown that under the quadratic $\mathscr{A}_{p,q}$-condition, the family is $\mathcal{R}$-bounded with the same quantitative bound as in (\ref{quantitative}) if and only if a quadratic testing condition for the whole family is satisfied. Our proof follows the broad outlines of $L^{2}$-theory but with additional complications coming from the general exponents. We also briefly outline the proof that if the dyadic shifts are of a  special form that arises naturally in the representation theorem concerning general Calder\'on-Zygmund operators, then a certain weakening of the $\mathscr{A}_{p,q}$-condition is sufficient.

It will be shown that this quadratic $\mathscr{A}_{p,q}$-constant is comparable to the constant in the ``two weight Stein's inequality'' for conditional expectations from $L^{p}$ into $L^{q}$ in the same way as the usual $A_{2}$-constant is related to boundedness of conditional expectations in weighted $L^{2}$.  We also construct an example showing that for $p>2$ or $1<q<2$ the $\mathscr{A}_{p,q}$-condition is in general stronger than the simple $A_{p,q}$-condition. Since they are equivalent in the case $1<p\leq 2 \leq q <\infty$, we deduce that the simple $A_{p,q}$-condition is  sufficient for the two weight Stein's inequality if and only $1<p\leq 2 \leq q <\infty$.

The two weight inequality for the dyadic square function was characterized in $L^{2}$ in terms of the Sawyer type testing and the $A_{2}$-condition in another paper by F. Nazarov, S. Treil and A. Volberg \cite{NTV2}. We use similar ideas as with the dyadic shifts and show that the two weight inequality for the dyadic square function holds from $L^{p}$ into $L^{q}$ if and only if the quadratic testing condition and the quadratic $\mathscr{A}_{p,q}$-condition hold, and we get a similar quantitative estimate as with the dyadic shifts. Here again we get the previous result as a special case when $p=q=2$. Our approach to the dyadic square function is inspired by the strategy in \cite{LL}, and similar steps appeared also in \cite{NTV2}.

\subsection*{Acknowledgements}
I am a member of the Finnish Centre of Excellence in Analysis and Dynamics Research. I am very grateful to my PhD advisor Tuomas Hyt\"onen for showing me the idea of quadratic testing and for suggesting the problem of this paper to me. This work is  part of my PhD project. I am also thankful for Timo H\"anninen for teaching me many facts about dyadic shifts.

\section{Set up and preliminaries}

We begin by specifying the basic notation and concepts we use. Two Radon measures $\si$ and $w$ on $\R^{N}$ are fixed. Most of the definitions below are made with respect to the measure $\si$, but it will be clear that they are defined similarly with respect to any Radon measure.

For any $1 \leq p \leq \infty$ the usual $L^{p}$-space with respect to the measure $\si$ is denoted by $L^{p}(\si)$. For a sequence $(f_{i})_{i=1}^{\infty}$ of Borel measurable functions we define
\begin{equation*}
\| (f_{i})_{i=1}^{\infty} \|_{L^{p}(\si;l^{2})}:=\Big( \int \Big( \sum_{i=1}^{\infty} | f_{i} |^{2}\Big)^{\frac{p}{2}} \ud \si \Big)^{\frac{1}{p}},
\end{equation*}
and the space $L^{p}(\si;l^{2})$ consists of those sequences $(f_{i})_{i=1}^{\infty}$ for which this norm is finite. All our functions will be real valued.

We fix a dyadic lattice $\D$ on $\R^{N}$. This means that $\D =\cup_{k\in \Z} \D_{k}$, where each $\D_{k}$ is a disjoint cover of $\R^{N}$ with cubes of the form $x+[0,2^{-k})^{N}, x\in \R^{N}$, and that for every $k \in \Z$ any cube $Q \in \D_{k}$ is a union of $2^{N}$-cubes in $\D_{k+1}$. 

If $ Q \in \D_{k}$, denote by $Q^{(1)}$ the unique cube in $\D_{k-1}$ that contains $Q$, and for any integer $r \geq 2$ define inductively  $Q^{(r)}:= (Q^{(r-1)})^{(1)}$. Write also $Q^{(0)}:=Q$. For $m=0,1,2,\dots$ the collection $ch^{(m)}(Q)$ consists of those $Q' \in \D$ such that $Q'^{(m)}=Q$, and we abbreviate $ch^{(1)}(Q)=:ch(Q)$.  The side length of a cube $Q \in \D_{k}$ is $l(Q):=2^{-k}$, and the volume $l(Q)^{N}$ is written as $|Q|$. 

\subsection*{Martingale decomposition}If $Q \in \D$ is any cube, the average of a locally $\si$-integrable function $f$ over $Q$ is denoted by
\begin{equation*}
\langle f \rangle^{\si}_{Q}:= \frac{1}{\si(Q)} \int_{Q} f \ud \si
\end{equation*}
with the understanding that $\langle f \rangle^\si_Q=0$ if $\si(Q)=0$. For two functions $f$ and $g$ we write $\langle f,g \rangle_\si:= \int f g \ud \si$ whenever the integral makes sense. The \emph{averaging} or \emph{conditional expectation operator} $\E_{k}, k \in \Z$, is defined as
\begin{equation*}
\E_{k}^{\si}f:= \sum_{Q \in \D_{k}} \langle f \rangle^{\si}_{Q}1_{Q}.
\end{equation*}

The martingale difference related to a cube $ Q \in \D$ is defined as
\begin{equation}\label{martingale difference}
\de^{\si}_{Q} f:= \sum_{Q' \in ch(Q)}  \langle f \rangle^{\si}_{Q'} 1_{Q'}- \langle f \rangle^{\si}_{Q} 1_{Q}.
\end{equation}

Let $(\varepsilon_{i})_{i=1}^{\infty}$ a sequence of independent random signs on some probability space $(\Omega, \mathbb{P})$. This means that the sequence is independent and $\mathbb{P}(\varepsilon_{i}=1)=\mathbb{P}(\varepsilon_{i}=-1)=1/2$ for all $i$.  We will use the Kahane-Khinchine inequality \cite{K} saying that for any Banach space $X$, any two exponents $1 \leq p,q< \infty$ and any elements $x_{1}, \dots, x_{M} \in X$ it holds that

\begin{equation}\label{KK}
\Big(\E \| \sum_{i=1}^{M} \varepsilon_{i}x_{i} \|_{X}^{q} \Big)^{\frac{1}{q}}  
\simeq_{p,q} \Big(\E \| \sum_{i=1}^{M} \varepsilon_{i}x_{i} \|_{X}^{p} \Big)^{\frac{1}{p}}, 
\end{equation}
where $\E$ refers to the expectation with respect to the random signs.

The notation $\simeq_{p,q}$ in (\ref{KK}) means that there exists a constant $C>0$ depending only on  $p$ and $q$ and not on $M$, $X$ nor on the elements $x_{i}$ such that if $A$ and $B$ denote the left and right hand sides of (\ref{KK}), respectively, then $C^{-1} B \leq A \leq C B$. The subscript refers to the information that the constant $C$ depends on, and is sometimes omitted. We use this kind of notation only if the constant $C$ does not depend on any relevant information in the situation, and no confusion should arise. Similarly $A \leq C B$ would be written as $A \lesssim B$.

Let $f \in L^{p}(\si)$ for some $1 < p < \infty$. Then we can do the martingale difference decomposition  
\begin{equation}\label{mart. dec.}
f = \sum_{Q \in \D_{l}} \langle f \rangle^{\si}_{Q} 1_{Q}
+\sum_{\begin{substack}{Q \in \D \\ l(Q) \leq 2^{-l}}\end{substack}} \de^{\si}_{Q} f,
\end{equation}
where $l \in \Z$ is any integer, and the series in (\ref{mart. dec.}) converges to $f$ in any order (that is, unconditionally). Burkholder's inequality
\begin{equation}\label{mart. norm}
\| f \|_{L^{p}(\si)}  
\simeq_{p} \Big\| \Big( \sum_{Q \in \D_{l}} | \langle f \rangle^{\si}_{Q} |^{2} 1_{Q}
+\sum_{\begin{substack}{Q \in \D \\ l(Q) \leq 2^{-l}}\end{substack}} | \de^{\si}_{Q} f |^{2}\Big)^{\frac{1}{2}} \Big\| _{L^{p}(\si)},
\end{equation}
implies that 
\begin{equation}\label{Burkholder's theorem}
\|f\|_{L^{p}(\si)} \simeq \E \Big\| \sum_{Q \in \D_{l}} \varepsilon_{Q} \langle f \rangle^{\si}_{Q} 1_{Q}
+\sum_{\begin{substack}{Q \in \D \\ l(Q) \leq 2^{-l}}\end{substack}} \varepsilon_{Q} \de^{\si}_{Q} f \Big\|_{L^{p}(\si)},
\end{equation}
where $\{\varepsilon_{Q}\}_{Q \in \D}$ is a collection of independent random signs. Burkholder's inequality (\ref{mart. norm}) was originally proved in \cite{B} in a little different situation.

From (\ref{Burkholder's theorem}) one can deduce with the Kahane-Khincine inequalities the following lemma for $L^{p}(\si;l^{2})$-norms. Below we shall also call equation (\ref{Burkholder}) Burkholder's inequality.

\begin{lem}
Let $1<p<\infty$ and suppose we have a sequence $ (f_{k})_{k=-\infty}^{\infty}\in L^{p}(\si;l^{2})$. Then we have the estimate
\begin{equation}\label{Burkholder}
\begin{split}
&\| (f_{k})_{k=-\infty}^{\infty} \|_{L^{p}(\si;l^{2})} \\
&\simeq_{p} \Big\| \Big( \sum_{k=-\infty}^{\infty} \sum_{Q \in \D_{l}} | \langle f_{k} \rangle^{\si}_{Q} |^{2} 1_{Q}
+\sum_{k=-\infty}^{\infty}\sum_{\begin{substack}{Q \in \D \\ l(Q) \leq 2^{-l}}\end{substack}} | \de^{\si}_{Q} f_{k} |^{2}\Big)^{\frac{1}{2}} \Big\| _{L^{p}(\si)},
\end{split}
\end{equation}
where $l \in \Z$ is any integer.
\end{lem} 

\begin{proof}
By monotone convergence we may assume that only finitely many functions $f_{k}$ are non zero. Furthermore, by the martingale convergence, we can suppose that for every $k$ there is only finitely many terms in the martingale decomposition of $f_{k}$. Thus the sums in the following computation are actually finite.

Let $\{\varepsilon_{k}\}_{k \in \Z}$ and $\{\varepsilon'_{Q}\}_{Q \in \D}$ be two sequences of independent  random signs on some distinct probability spaces, and we write $\E$ and $\E'$ for the corresponding expectations. Then we compute with the Kahane-Khinchine inequalities and equation (\ref{Burkholder's theorem}) that
\begin{equation}\label{vector norm}
\begin{split}
&\| (f_{k})_{k=-\infty}^{\infty} \|_{L^{p}(\si;l^{2})} ^{p}
=\Big\| \Big(\E\Big|\sum_{k=-\infty}^{\infty} \varepsilon_{k}f_{k}\Big|^{2}\Big)^{\frac{1}{2}} \Big\|_{L^{p}(\si)}^{p} \\
&\simeq  \E \int_{\R^{N}}\Big|\sum_{k=-\infty}^{\infty} \varepsilon_{k}f_{k}\Big|^{p}\ud \si \\ 
&\simeq  \E  \E' \int_{\R^{N}}\Big|\sum_{k=-\infty}^{\infty} \sum_{Q \in \D_{l}} \varepsilon_{k}\varepsilon'_{Q} \langle f_{k} \rangle^{\si}_{Q}  1_{Q}
+\sum_{k=-\infty}^{\infty}\sum_{\begin{substack}{Q \in \D \\ l(Q) \leq 2^{-l}}\end{substack}} \varepsilon_{k}\varepsilon'_{Q} \de^{\si}_{Q} f_{k} \Big|^{p} \ud \si.
\end{split}
\end{equation}

If $\{c_{k,Q}\}_{k \in \Z, Q\in \D}$ is any doubly indexed finitely non zero set of real numbers, then
\begin{equation}\label{double Rad}
\begin{split}
&\E  \E'\Big|\sum_{k=-\infty}^{\infty} \sum_{Q \in \D} \varepsilon_{k}\varepsilon'_{Q}c_{k,Q} \Big|^{p} 
=\E  \E'\Big| \sum_{Q \in \D}\varepsilon'_{Q} \sum_{k=-\infty}^{\infty} \varepsilon_{k}c_{k,Q} \Big|^{p} \\
&\simeq \E\Big( \E' \Big|\sum_{Q \in \D}\varepsilon'_{Q} \sum_{k=-\infty}^{\infty} \varepsilon_{k}c_{k,Q} \Big|^{2}\Big)^{\frac{p}{2}}
= \E\Big( \E' \Big|\sum_{k=-\infty}^{\infty} \varepsilon_{k} \sum_{Q \in \D}  \varepsilon'_{Q}c_{k,Q} \Big|^{2}\Big)^{\frac{p}{2}} \\
&\simeq  \Big(\E \E' \Big|\sum_{k=-\infty}^{\infty} \varepsilon_{k} \sum_{Q \in \D}  \varepsilon'_{Q}c_{k,Q} \Big|^{2}\Big)^{\frac{p}{2}}
= \Big(\sum_{k=-\infty}^{\infty} \sum_{Q \in \D} |c_{k,Q}|^{2} \Big)^{\frac{p}{2}}.
\end{split}
\end{equation}

Using (\ref{double Rad}) in (\ref{vector norm}) we get the estimate we wanted.

\end{proof}

\subsection*{Principal cubes and Carleson's embedding theorem} We will need the construction of principal cubes. More precisely, suppose $f \in L^{1}_{loc}(\si)$ and take some cube $Q_{0} \in \D$. Set $\s_{0}=\{Q_{0}\}$, and assume that $\s_{0}, \dots, \s_{k}$ are defined for some non negative integer $k$. Then, for $S \in \s_{k}$,  let $ch_{\s}(S)$ consist of the maximal cubes $S' \in \D$ such that $S' \subset S$ and
\begin{equation*}
\langle |f| \rangle_{S'}^{\si} > 2 \langle |f| \rangle_{S}^{\si}.
\end{equation*}
Set $\s_{k+1}:= \cup_{S \in \s_{k}} ch_{\s}(S)$ and 
\begin{equation*}
\s:= \bigcup_{k=0}^{\infty}\s_{k}.
\end{equation*}

Now for every cube $Q\in \D, Q\subset Q_{0}$, there exists a unique smallest $S \in \s$, denoted by $\pi_{\s}Q=S$, that contains $Q$, and it follows from the construction that $\langle |f|\rangle^{\si}_{Q} \leq 2 \langle |f|\rangle^{\si}_{S}.$ 

Let $\gamma \in (0,1)$. We say that a collection $\D_{0} \subset \D$ is $\gamma$-sparse if there exist pairwise disjoint measurable sets $E(Q) \subset Q$, $ Q \in \D_{0}$, such that   $\si(E(Q)) \geq \gamma\si(Q)$ for all $ Q \in \D_{0}$. The collection $\s$ of stopping cubes constructed above is a $\frac{1}{2}$-sparse collection, which is seen by defining $E(S):= S \setminus \bigcup_{S' \in ch_\s(S)}S'$, $S \in \s$. Related to these sparse families we shall use the following form of \emph{Carleson's embedding theorem}:

\begin{lem}

Suppose $1<p<\infty$, $\gamma \in (0,1)$ and $(f_{k})_{k=1}^{\infty} \subset L^{p}(\si;l^{2})$. For each $k$ let $\s_{k}$ be any $\gamma$-sparse collection. Then
\begin{equation}\label{vector Carleson}
\Big\|\Big(\sum_{k=1}^{\infty}\sum_{S \in \s_{k}} \big(\langle f_{k}\rangle^{\si}_{S}\big)^{2}1_{S}\Big)^{\frac{1}{2}} \Big\|_{L^{p}(\si)} 
\lesssim_{\gamma,p} \Big\|\Big(\sum_{k=1}^{\infty} f_{k}^{2}\Big)^{\frac{1}{2}} \Big\|_{L^{p}(\si)}. 
\end{equation}

\end{lem}

\begin{proof}

Let $M^{d}_{\si}$ be the dyadic maximal function defined for any Borel measurable $f$ by
\begin{equation*}
M^{d}_{\si}(f)= \sup_{Q \in \D} 1_{Q}\langle |f| \rangle^{\si}_{Q}.
\end{equation*}
For any $k$ and $S \in \s_{k}$ denote again by $E_k(S)$ the measurable subset of $S$ such that $\si(E_k(S)) \geq \gamma\si(S)$ and $E_k(S')\cap E_k(S)=\emptyset$ for any other $S' \in \s_{k}$.

To prove (\ref{vector Carleson}), assume  without loss of generality that every $f_{k}$ is non negative. We want to argue by duality, and for that purpose let $\{g_{k,S}: k =1,2,\dots, \ S \in \s_{k}\}$ be any finitely non zero collection of $L^{p'}(\si)$ functions ($p'$ denotes the H\"older conjugate exponent to $p$). Then
\begin{equation*}
\begin{split}
&\sum_{k=1}^{\infty}\sum_{S \in \s_{k}}\int  \langle f_{k}\rangle^{\si}_{S}1_{S}g_{k,S} \ud \sigma
\leq \gamma^{-1} \sum_{k=1}^{\infty}\sum_{S \in \s_{k}}\langle f_{k}\rangle^{\si}_{S}\langle g_{k,S}\rangle^{\si}_{S}\si(E_k(S)) \\
&\leq \gamma^{-1}\sum_{k=1}^{\infty}\sum_{S \in \s_{k}} \int M^{d}_{\si}(f_{k}) M^{d}_{\si}(g_{k,S})1_{E_k(S)} \ud \si \\
&\leq \gamma^{-1}\Big\|\Big( \sum_{k=1}^{\infty}\sum_{S \in \s_{k}} \big(M^{d}_{\si}(f_{k})\big)^{2}1_{E_k(S)}\Big)^{\frac{1}{2}}\Big\|_{L^{p}(\si)} \\
&\times \Big\|\Big( \sum_{k=1}^{\infty}\sum_{S \in \s_{k}} \big(M^{d}_{\si}(g_{k,S})\big)^{2}1_{E_k(S)}\Big)^{\frac{1}{2}}\Big\|_{L^{p'}(\si)}.
\end{split}
\end{equation*}

Since for a fixed $k$ the sets $E_k(S), S \in \s_{k}$, are pairwise disjoint, the first factor satisfies
\begin{equation*}
\begin{split}
&\Big\|\Big( \sum_{k=1}^{\infty}\sum_{S \in \s_{k}} \big(M^{d}_{\si}(f_{k})\big)^{2}1_{E_k(S)}\Big)^{\frac{1}{2}}\Big\|_{L^{p}(\si)} \\
&\leq \Big\|\Big( \sum_{k=1}^{\infty} \big(M^{d}_{\si}(f_{k})\big)^{2}\Big)^{\frac{1}{2}}\Big\|_{L^{p}(\si)} 
\lesssim_{p} \Big\|\Big( \sum_{k=1}^{\infty}f_{k}^{2}\Big)^{\frac{1}{2}}\Big\|_{L^{p}(\si)},
\end{split}
\end{equation*}
where in the last step we used the dyadic Fefferman-Stein inequality \cite{FS}. In the second factor we may just omit the indicators $1_{E_k(S)}$ and apply the Fefferman-Stein inequality again. These estimates prove (\ref{vector Carleson}).

\end{proof}

\subsection*{Stein's inequality}
Let $(f_{k})_{k=-\infty}^{\infty} \in L^{p}(\si;l^{2}), 1<p<\infty$, be a sequence of functions. \emph{Stein's inequality}, which originally appeared in \cite{St}, says that
\begin{equation}\label{Stein1}
\|(E^{\si}_{k}f_{k})_{k=-\infty}^{\infty} \|_{L^{p}(\si;l^{2})} 
\lesssim_{p}\|(f_{k})_{k=-\infty}^{\infty} \|_{L^{p}(\si;l^{2})}.
\end{equation}
This can equivalently be formulated by saying that for any set $\{f_{Q}\}_{Q \in \D}$, where each $f_{Q}$ is a locally $\si$-integrable function, the inequality
\begin{equation}\label{Stein}
\Big\|\Big(\sum_{Q \in \D}   \big(\langle f_{Q} \rangle^{\si}_{Q}\big)^{2}1_{Q}\Big)^{\frac{1}{2}}\Big\|_{L^{p}(\si)}
\lesssim_{p} \Big\|\Big(\sum_{Q \in \D} f_{Q}^{2}1_{Q}\Big)^{\frac{1}{2}}\Big\|_{L^{p}(\si)}
\end{equation}
holds. Note that (\ref{Stein1}) follows also from the dyadic Fefferman-Stein inequality that was used in the proof of Carleson's embedding theorem.

\section{The quadratic $\mathscr{A}_{p,q}$-condition}\label{square Ap condition}
In this section we introduce the quadratic $\mathscr{A}_{p,q}$-condition and investigate its relation with the Muckenhoupt type $A_{p,q}$-condition. Here the exponents satisfy $1<p,q<\infty$. The quadratic $\mathscr{A}_{p,q}$-condition will be used in the characterization of two weight inequalities for the dyadic square function and the dyadic shifts.

The measures $\si$ and $w$ are said to satisfy the simple- or Muckenhoupt type $A_{p,q}$-condition if
\begin{equation}\label{simple A_p}
(\si,w)_{p,q}:= \sup_{Q \in \D}\frac{\si(Q)^{\frac{1}{p'}}w(Q)^{\frac{1}{q}}}{|Q|} < \infty.
\end{equation}
If $p=q$ we write just $A_{p}$ instead. 

The measures $\si$ and $w$ are said to satisfy the quadratic $\mathscr{A}_{p,q}$-condition if for every collection $\{a_{Q}\}_{Q\in \D}$ of real numbers the inequality
\begin{equation}\label{square Ap}
\Big\|\Big(\sum_{Q \in \D} \Big(a_{Q}\frac{\si(Q)}{|Q|}\Big)^{2}1_{Q}\Big)^{\frac{1}{2}}\Big\|_{L^{q}(w)}
\leq [\si,w]_{p,q} \Big\|\Big(\sum_{Q \in \D} a_{Q}^{2}1_{Q}\Big)^{\frac{1}{2}}\Big\|_{L^{p}(\si)}
\end{equation}
holds, where $[\si,w]_{p,q} \in [0, \infty)$ is the best possible constant. We also write $[\si,w]_{p,q}<\infty$ to mean that the condition holds, and $[\si,w]_{p,q}=\infty$ to mean that it does not hold. It is clear that $(\si,w)_{p,q} \leq [\si,w]_{p,q}$, which follows by taking only one term in the sums in (\ref{square Ap}).

\begin{lem}
Let $1<p,q<\infty$. The quadratic $\mathscr{A}_{p,q}$-condition is symmetric in the sense that $[\si,w]_{p,q} \simeq [w,\si]_{q',p'}$.
\end{lem}

\begin{proof}
Choose any (finitely non zero) collection $\{a_{Q}\}_{Q\in \D}$ of real numbers, and let also $\{f_{Q}\}_{Q \in \D}$ be a collection of $L^{p}(\si)$-functions. Then 
\begin{equation*}
\begin{split}
&\int\sum_{Q \in \D} a_{Q}\frac{w(Q)}{|Q|}1_{Q} f_{Q} \ud \sigma
= \int \sum_{Q \in \D} a_{Q} \frac{ \int_{Q}f_{Q} \ud \sigma}{|Q|}1_{Q} \ud w \\
& \leq \Big\|\Big(\sum_{Q \in \D} a_{Q}^{2}1_{Q}\Big)^{\frac{1}{2}}\Big\|_{L^{q'}(w)}
\Big\|\Big(\sum_{Q \in \D} \Big( \langle |f_{Q}|\rangle^{\si}_{Q} \frac{\si(Q)}{|Q|}\Big)^{2}1_{Q}\Big)^{\frac{1}{2}}\Big\|_{L^{q}(w)} \\
&\leq [\si,w]_{p,q} \Big\|\Big(\sum_{Q \in \D} a_{Q}^{2}1_{Q}\Big)^{\frac{1}{2}}\Big\|_{L^{q'}(w)}
\Big\|\Big(\sum_{Q \in \D} \big(\langle |f_{Q}|\rangle^{\si}_{Q}\big)^{2}1_{Q}\Big)^{\frac{1}{2}}\Big\|_{L^{p}(\si)} \\
& \lesssim [\si,w]_{p,q} \Big\|\Big(\sum_{Q \in \D} a_{Q}^{2}1_{Q}\Big)^{\frac{1}{2}}\Big\|_{L^{q'}(w)}
\Big\|\Big(\sum_{Q \in \D} |f_{Q}1_{Q}|^{2}\Big)^{\frac{1}{2}}\Big\|_{L^{p}(\si)},
\end{split}
\end{equation*} 
where in the last step we used Stein's inequality. By duality this shows that $[w,\si]_{q',p'} \lesssim [\si,w]_{p,q}.$
\end{proof}

For $1<p,q<\infty$ a two weight version of Stein's inequality (\ref{Stein}) can be formulated as
\begin{equation}\label{two weight Stein}
\Big\|\Big(\sum_{Q \in \D} \Big(  \frac{\int_{Q}f_{Q} \ud \si}{|Q|}\Big)^{2}1_{Q}\Big)^{\frac{1}{2}}\Big\|_{L^{q}(w)}
\leq \mathscr{S} \Big\|\Big(\sum_{Q \in \D} f_{Q}^{2}1_{Q}\Big)^{\frac{1}{2}}\Big\|_{L^{p}(\si)},
\end{equation}
where $\{f_{Q}\}_{Q \in \D}$ is again a collection of locally $\si$-integrable functions, and $\mathscr{S}=\mathscr{S}(\si,w,p,q)$ denotes the smallest possible constant with the understanding that it may be infinite.

\begin{lem}\label{square&Stein}
The best constant $\mathscr{S}=\mathscr{S}(\si,w,p,q)$ in (\ref{two weight Stein}) satisfies $\mathscr{S} \simeq [\si,w]_{p,q}$.
\end{lem}

\begin{proof}
That $[\si,w]_{p,q} \leq \mathscr{S}(\si,w,p,q)$ follows from (\ref{two weight Stein}) with the special functions $f_{Q}=a_{Q}1_{Q}$, where $a_{Q} \in \R$. To see that $\mathscr{S}(\si,w,p,q) \lesssim [\si,w]_{p,q}$, choose any  set $\{f_{Q}\}_{Q \in \D}$ of locally $\si$-integrable functions. Then
\begin{equation*}	
\begin{split}
&LHS(\ref{two weight Stein}) = \Big\|\Big(\sum_{Q \in \D} \Big( \langle f_{Q}\rangle^{\si}_{Q} \frac{\si(Q)}{|Q|}\Big)^{2}1_{Q}\Big)^{\frac{1}{2}}\Big\|_{L^{q}(w)}\\
&\leq [\si,w]_{p,q} \Big\|\Big(\sum_{Q \in \D}  \big(\langle f_{Q}\rangle^{\si}_{Q}\big)^{2} 1_{Q}\Big)^{\frac{1}{2}}\Big\|_{L^{p}(\si)} 
 \lesssim [\si,w]_{p,q} \Big\|\Big(\sum_{Q \in \D}  f_{Q}^{2} 1_{Q}\Big)^{\frac{1}{2}}\Big\|_{L^{p}(\si)},
\end{split}
\end{equation*}
where we used Stein's inequality (\ref{Stein}) in the last step. Hence also $[\si,w]_{p,q} \lesssim \mathscr{S}(\si,w,p,q)$.
\end{proof}

The next lemma shows that the quadratic $\mathscr{A}_{p,q}$-condition is actually equivalent with the simple $A_{p,q}$-condition in the case $1<p\leq 2 \leq q <\infty$, and a similar remark will apply to the quadratic testing conditions below.

\begin{lem}\label{simple implies square}
If $1< p\leq 2 \leq q<\infty$, then $[\si,w]_{p,q}=(\si,w)_{p,q}$.
\end{lem}

\begin{proof}
This follows from the fact that $L^{p}$-spaces have certain \emph{type} and \emph{cotype} properties. For our purposes it is not necessary to define these in general, but it suffices to note that for any sequence $(f_{k})_{k=1}^{\infty} \subset L^{p}(\si;l^{2}), 1<p\leq 2$, it holds that
\begin{equation}\label{cotype}
\Big\|\Big(\sum_{k =1}^{\infty} f_{k}^{2}\Big)^{\frac{1}{2}}\Big\|_{L^{p}(\si)}
\geq \Big( \sum_{k =1}^{\infty} \|f_{k}\|_{L^{p}(\si)}^{2} \Big)^{\frac{1}{2}},
\end{equation}
and for any sequence  $(g_{k})_{k=1}^{\infty} \subset L^{q}(\si;l^{2}), 2\leq q<\infty,$ 
\begin{equation}\label{type}
\Big\|\Big(\sum_{k =1}^{\infty} g_{k}^{2}\Big)^{\frac{1}{2}}\Big\|_{L^{q}(\si)}
\leq \Big( \sum_{k =1}^{\infty} \|g_{k}\|_{L^{q}(\si)}^{2} \Big)^{\frac{1}{2}}.
\end{equation} 
Of course these inequalities are independent of the measure.

Suppose then that the simple $A_{p,q}$-condition holds with $1<p\leq 2\leq q<\infty$,  and let $\{a_{Q}\}_{Q \in \D}\subset \R$ be any collection. Then
\begin{equation}\label{simple strong}
\begin{split}
&\Big\|\Big(\sum_{Q \in \D} \Big(a_{Q}\frac{\si(Q)}{|Q|}\Big)^{2}1_{Q}\Big)^{\frac{1}{2}}\Big\|_{L^{q}(w)}
\leq \Big( \sum_{Q \in \D} \Big\|a_{Q}\frac{\si(Q)}{|Q|}1_{Q}\Big\|_{L^{q}(w)}^{2} \Big)^{\frac{1}{2}} \\
&\leq (\si,w)_{p,q} \Big( \sum_{Q \in \D} \big\|a_{Q}1_{Q}\big\|_{L^{p}(\si)}^{2} \Big)^{\frac{1}{2}} 
 \leq (\si,w)_{p,q}\Big\|\Big(\sum_{Q \in \D} a_{Q}^{2}1_{Q}\Big)^{\frac{1}{2}}\Big\|_{L^{p}(\si)},
\end{split}
\end{equation}
and thus $[\si,w]_{p,q} \leq (\si,w)_{p,q}$.
\end{proof}

\section{The dyadic square function}
In this section we consider the dyadic square function. Let $\{b_{Q}\}_{Q \in \D}$ be a collection of real numbers. For a locally Lebesgue integrable function the generalized dyadic square function is defined by
\begin{equation*}
S_{b}(f):=\Big(\sum_{Q \in \D}\big( b_{Q}\de_{Q}f\big)^{2}\Big)^{\frac{1}{2}},
\end{equation*} 
where $\de_{Q}f$ is the usual martingale difference related to the cube $Q$ as in (\ref{martingale difference}), but with respect to the Lebesgue measure. The ``generalized'' here refers to the coefficients $b_{Q}$, and the usual dyadic square function corresponds to the case $b_{Q}=1$ for all $Q \in \D$.

Now we are interested in the two weight estimate for this operator. Namely, we fix two exponents $1<p,q<\infty$ and want to characterize when there exists a constant $C\geq 0$ such that the inequality
\begin{equation}\label{problem for square}
\Big \| \Big(\sum_{Q\in\D} \big(b_{Q}\de_{Q}(f\si)\big)^{2}\Big)^{\frac{1}{2}} \Big\|_{L^{q}(w)} \leq C \|f\|_{L^{p}(\si)}
\end{equation}
holds for all $f \in L^{p}(\si)$. Here $\de_{Q}(f\si)$ is understood as
\begin{equation*}
\de_{Q} (f\si):= \sum_{Q' \in ch(Q)}  \frac{\int_{Q'} f \ud\si}{|Q'|} 1_{Q'}- \frac{\int_{Q}f \ud\si}{|Q|} 1_{Q}.
\end{equation*}
Denote by $S^{\si}_{b}$ the operator defined for locally $\si$-integrable functions by 
\begin{equation*}
S^{\si}_{b}(f):=\Big(\sum_{Q \in \D}\big( b_{Q}\de_{Q}(f \si)\big)^{2}\Big)^{\frac{1}{2}},
\end{equation*}
and define also for all $Q \in \D$ the localized version
\begin{equation*}
S^{\si}_{b,Q}(f):=\Big(\sum_{\begin{substack}{Q' \in \D: \\ Q' \subset Q}\end{substack}}\big( b_{Q'}\de_{Q'}(f \si)\big)^{2}\Big)^{\frac{1}{2}}.
\end{equation*}

If $u$ and $v$ are two weight functions on $\R$, that is, positive Borel functions, and  $p=q=2$, we have the result from \cite{NTV2} saying that
\begin{equation}\label{problem for square p=2}
\Big \| S_{b}(fu) \Big\|_{L^{2}(v)} \leq C \|f\|_{L^{2}(u)}
\end{equation}
holds if and only if there exists a constant $C'$ such that
\begin{equation}\label{Sawyer for square} 
\| S_{b}(1_{I}u)\|_{L^{2}(v)} \leq C' \|1_{I}\|_{L^{2}(u)}
\end{equation}
holds for all $I \in \D$. Also in this case the best constants in (\ref{problem for square p=2}) and (\ref{Sawyer for square}) satisfy $C' \simeq C$. Actually a bit more was shown, namely that the two weight inequality holds if and only if a Muckenhoupt type condition for the measures and a localized testing condition hold.

Here we are going to give a characterization for the inequality (\ref{problem for square}) with any exponents $1<p,q<\infty$.  This will be done in terms of a quadratic testing condition and the quadratic $\mathscr{A}_{p,q}$-condition introduced in the last section, and in the case $p=q=2$ the theorem reduces to the result from \cite{NTV2}.

We say that the operator $S_{b}^{\si}$ satisfies the \emph{global quadratic testing condition} (with respect to $p$ and $q$) if there exists a constant $C$ such that for every collection $\{a_{Q}\}_{Q\in \D} \subset \R$ the inequality
\begin{equation}\label{global square}
\Big\|\Big(\sum_{Q \in \D}S^{\si}_{b}(a_{Q}1_{Q})^{2}\Big)^{\frac{1}{2}}\Big\|_{L^{q}(w)}
\leq C \Big\|\Big(\sum_{Q \in \D} a_{Q}^{2}1_{Q}\Big)^{\frac{1}{2}}\Big\|_{L^{p}(\si)}
\end{equation}
holds. The operator $S^{\si}_{b}$ is said to satisfy the \emph{local quadratic testing condition} if it similarly satisfies estimate
\begin{equation}\label{local square}
\Big\|\Big(\sum_{Q \in \D}S_{b,Q}^{\si}(a_{Q}1_{Q})^{2}\Big)^{\frac{1}{2}}\Big\|_{L^{q}(w)}
\leq C \Big\|\Big(\sum_{Q \in \D} a_{Q}^{2}1_{Q}\Big)^{\frac{1}{2}}\Big\|_{L^{p}(\si)}.
\end{equation}
Of course it is equivalent to assume that these inequalities hold for all finitely non zero collections $\{a_{Q}\}_{Q\in \D}$.

We shall modify the quadratic $\mathscr{A}_{p,q}$-conditions according to the coefficients $b_{Q}$. The measures satisfy the $\mathscr{A}_{p,q}^{b}$-condition if for every collection $\{a_{Q}\}_{Q\in \D}$ of real numbers the inequality
\begin{equation}\label{square Ap, generalized}
\Big\|\Big(\sum_{Q \in \D} \Big(a_{Q}b_{Q}\frac{\si(Q)}{|Q|}\Big)^{2}1_{Q}\Big)^{\frac{1}{2}}\Big\|_{L^{q}(w)}
\leq [\si,w]^{b}_{p,q} \Big\|\Big(\sum_{Q \in \D} a_{Q}^{2}1_{Q}\Big)^{\frac{1}{2}}\Big\|_{L^{p}(\si)}
\end{equation}
holds, where again $[\si,w]^{b}_{p,q}$ denotes the best possible constant.

Now we can state the two weight theorem for the dyadic square function as follows:

\begin{theorem}\label{two weight square}
Let $1<p,q<\infty$. The dyadic square function $S^{\si}_{b}$ satisfies the two weight inequality (\ref{problem for square})
if and only if it satisfies the global quadratic testing condition (\ref{global square}) and if and only if it satisfies the local quadratic testing condition (\ref{local square}) and the quadratic $\mathscr{A}^{b}_{p,q}$-condition (\ref{square Ap, generalized}) holds.

In this case the best constant $\|S^{\si}_{b}\|$ in (\ref{problem for square}) satisfies $\|S^{\si}_{b}\| \simeq \mathfrak{S}_{glob} \simeq \mathfrak{S}_{loc}+[\si,w]^{b}_{p,q}$, where $\mathfrak{S}_{glob}$ and $\mathfrak{S}_{loc}$ are the best possible constants in (\ref{global square}) and (\ref{local square}), respectively.

\end{theorem}

Let us discuss the case $p=q=2$, or more generally the case $1<p\leq 2\leq q<\infty$. Similarly as we noted above in Lemma \ref{simple implies square}, then the $\mathscr{A}^{b}_{p,q}$-condition is equivalent to assuming
\begin{equation*}
\sup_{Q \in \D} |b_{Q}| \frac{\si(Q)^{\frac{1}{p'}}w(Q)^{\frac{1}{q}}}{|Q|} \lesssim 1.
\end{equation*}
The same kind of computation shows that the quadratic testing conditions are equivalent to the corresponding Sawyer type testing conditions. For example considering the global testing (\ref{global square}), this means that it is enough to assume just 
\begin{equation*} 
\Big \| S^{\si}_{b}(1_{Q})\Big\|_{L^{q}(w)} \leq C \si(Q)^{\frac{1}{p}}
\end{equation*}
uniformly for all $Q \in \D$.

With these facts Theorem \ref{two weight square} reduces to the result proved in \cite{NTV2} when $p=q=2$.

\begin{proof}[Proof of Theorem \ref{two weight square}]

We begin by showing that the global, and hence also the local testing condition is a necessary consequence of the two weight inequality (\ref{problem for square}). Then we show that the global testing implies the quadratic $\mathscr{A}^{b}_{p,q}$-condition. The main part of the proof is in showing that the local  testing and the $\mathscr{A}^{b}_{p,q}$-condition are also sufficient for (\ref{problem for square}).

\subsection*{Necessity of the testing conditions}  This is very much like a classical theorem of Marcinkiewicz and Zygmund \cite{MZ}, which says that bounded linear operators in $L^{p}$-spaces have an extension into a vector valued situation. 
Choose a sequence $(f_{k})_{k=1}^{l} \subset L^{p}(\si)$ and let $(\varepsilon_{k})_{k=1}^{l}$ be a sequence of independent random signs.
Then we compute with the Kahane-Khinchine inequalities that

\begin{equation}\label{vector extension}
\begin{split}
&\Big\|\Big(\sum_{k=1}^{l}|S^{\si}_{b}(f_{k})|^{2}\Big)^{\frac{1}{2}}\Big\|_{L^{q}(w)}
=\Big \| \Big( \sum_{Q \in \D} \sum_{k=1}^{l} |b_{Q}\de_{Q}(f_{k}\si)|^{2} \Big)^{\frac{1}{2}} \Big\|_{L^{q}(w)} \\
&=\Big \| \Big(  \sum_{Q \in \D} \E \Big| \sum_{k=1}^{l} \varepsilon_{k}b_{Q}\de_{Q}(f_{k}\si)\Big|^{2} \Big)^{\frac{1}{2}} \Big\|_{L^{q}(w)} \\
& =\Big \| \Big(\E\Big\|\big\{\sum_{k=1}^{l} \varepsilon_{k}b_{Q}\de_{Q}(f_{k}\si)\big\}_{Q \in \D}\Big\|_{l^{2}}^{2}\Big)^{\frac{1}{2}}\Big\|_{L^{q}(w)} \\
&\simeq \Big(\E \Big \| \Big\|\big\{\sum_{k=1}^{l} \varepsilon_{k}b_{Q}\de_{Q}(f_{k}\si)\big\}_{Q \in \D}\Big\|_{l^{2}}\Big\|^{q}_{L^{q}(w)}\Big)^{\frac{1}{q}} \\
&\simeq \E \Big \| \Big\|\big\{\sum_{k=1}^{l} \varepsilon_{k}b_{Q}\de_{Q}(f_{k}\si)\big\}_{Q \in \D}\Big\|_{l^{2}}\Big\|_{L^{q}(w)},
\end{split}
\end{equation}
where at the first ``$\simeq$'' we used Kahane-Khinchine inequality in $l^{2}$ and at the second in $L^{q}(w;l^{2})$. Linearity of the martingale differences and the assumed two weight inequality (\ref{problem for square}) imply
\begin{equation}\label{vector extension continued}
\begin{split}
&RHS(\ref{vector extension}) = \E \Big\| S^{\si}_{b}\big(\sum_{k=1}^{l} \varepsilon_{k}f_{k}\big)\Big\|_{L^{q}(w)}\\
& \leq \|S^{\si}_{b}\|\E\Big\|\sum_{k=1}^{l} \varepsilon_{k}f_{k}\Big\|_{L^{p}(\si)}
\simeq \|S^{\si}_{b}\|\Big\|\Big(\sum_{k=1}^{l} f_{k}^{2}\Big)^{\frac{1}{2}}\Big\|_{L^{p}(\si)},
\end{split}
\end{equation}
where at the ``$\simeq$'' we used   Kahane-Khinchine inequality first in $L^{p}(\si)$ and then in $\R$. With (\ref{vector extension}) and (\ref{vector extension continued}) it is seen that the two weight inequality (\ref{problem for square}) implies the global quadratic testing condition (\ref{global square}).

\subsection*{Global testing implies the $\mathscr{A}^{b}_{p,q}$-condition}

For any $Q \in \D$ let $\{Q_{k}\}_{k=1}^{2^{N}}$ be its dyadic children. If $Q \in \D$ and $k \in \{1,\dots,2^{N}\}$, then 

\begin{equation*}
\frac{\si(Q_{k})}{|Q_{k}|} \lesssim |\de_{Q}(1_{Q_{k}}\si)(x) |
\end{equation*} 
for any $x \in Q$, and thus 
\begin{equation*}
 |a_{Q}b_{Q}|\frac{\si(Q_{k})}{|Q_{k}|}1_{Q} \lesssim  S_{b,Q}^{\si}(a_{Q}1_{Q_{k}}).
\end{equation*}
This leads to 
\begin{equation*}
\begin{split}
&\Big\|\Big(\sum_{Q \in \D} \Big(a_{Q}b_{Q}\frac{\si(Q_{k})}{|Q_{k}|}\Big)^{2}1_{Q}\Big)^{\frac{1}{2}}\Big\|_{L^{q}(w)}
\lesssim \Big\|\Big(\sum_{Q \in \D}S_{b,Q}^{\si}(a_{Q}1_{Q_{k}})^{2}\Big)^{\frac{1}{2}}\Big\|_{L^{q}(w)} \\
& \leq \Big\|\Big(\sum_{Q \in \D}S_{b}^{\si}(a_{Q}1_{Q_{k}})^{2}\Big)^{\frac{1}{2}}\Big\|_{L^{q}(w)}
\leq \mathfrak{S}_{glob} \Big\|\Big(\sum_{Q \in \D} a_{Q}^{2}1_{Q_{k}}\Big)^{\frac{1}{2}}\Big\|_{L^{p}(\si)}\\
&\leq \mathfrak{S}_{glob} \Big\|\Big(\sum_{Q \in \D} a_{Q}^{2}1_{Q}\Big)^{\frac{1}{2}}\Big\|_{L^{p}(\si)}.
\end{split}
\end{equation*}
Since
\begin{equation*}
\Big(\sum_{Q \in \D} \Big(a_{Q}b_{Q}\frac{\si(Q)}{|Q|}\Big)^{2}1_{Q}\Big)^{\frac{1}{2}}
\leq \sum_{k=1}^{2^{N}} \Big(\sum_{Q \in \D} \Big(a_{Q}b_{Q}\frac{\si(Q_{k})}{|Q_{k}|}\Big)^{2}1_{Q}\Big)^{\frac{1}{2}},
\end{equation*}
we get $[\si,w]^{b}_{p,q} \lesssim \mathfrak{S}_{glob}$.

\subsection*{Sufficiency of the local testing and the $\mathscr{A}^{b}_{p,q}$-condition} Now we turn to the main part of the theorem, which consists of showing that the local testing and the $\mathscr{A}^{b}_{p,q}$-condition are sufficient for the estimate (\ref{problem for square}). To this end, fix a function $f \in L^{p}(\si)$. We can assume here that there are only finitely many non zero coefficients $b_{Q}$ in the definition of $S^{\si}_{b}$, and we prove a bound that is independent of this finite number. Of course the original local testing condition implies the same condition for this ``truncated'' square function.

There are at most $2^{N}$ increasing sequences $Q^{i}_{1} \subsetneq Q^{i}_{2} \subsetneq \dots$, $i  =1,\dots,j\leq 2^{N}$, of dyadic cubes in $\D$ such that 
\begin{equation}\label{quadrants}
\R^{N}= \bigcup_{i=1}^{j} \bigcup_{k=1}^{\infty} Q^{i}_{k}
\end{equation}
and 
\begin{equation*}
\bigcup_{k=1}^{\infty} Q^{i}_{k} \cap \bigcup_{k=1}^{\infty} Q^{i'}_{k} =\emptyset
\end{equation*}
if $i \not= i'$. It follows from the properties of dyadic systems that for every cube $Q \in \D$ there exists $i \in \{1,\dots,j\}$ such that $Q \subset \cup_{k=1}^{\infty} Q^{i}_{k}$.

Since there are only finitely many non zero $b_{Q}$s, we can choose indices $k_{1}, \dots, k_{j}$ such that if $b_{Q} \not=0$, then $Q \subset \cup_{i=1}^{j} Q^{i}_{k_{i}}$, and we write $\tilde{Q}_{i}:=Q^{i}_{k_{i}}$. Thus we can assume that the function $f$ is supported on $\cup_{i=1}^{j} \tilde{Q}_{i}$. Since $S_{b}^{\si}f=\sum_{i=1}^{j}S_{b}^{\si}(1_{\tilde{Q}_{i}}f)$, it is enough to bound each of these separately.

The choice of the cubes $\tilde{Q}_{i}$ implies that  $S^{\si}_{b}(1_{\tilde{Q}_{i}})=S_{b,\tilde{Q}_{i}}^{\si}(1_{\tilde{Q}_{i}})$, and thus 
\begin{equation*}
\begin{split}
&\Big\| \langle f \rangle^{\si}_{\tilde{Q}_{i}}S_{b}^{\si}(1_{\tilde{Q}_{i}}) \Big\|_{L^{q}(w)}
= \Big\|  \langle f \rangle^{\si}_{\tilde{Q}_{i}}S_{b,\tilde{Q}_{i}}^{\si}(1_{\tilde{Q}_{i}})\Big\|_{L^{q}(w)} \\
&\leq \mathfrak{S}_{loc} \|  \langle f \rangle^{\si}_{\tilde{Q}_{i}}1_{\tilde{Q}_{i}}\|_{L^{p}(\si)}
\leq  \mathfrak{S}_{loc} \|1_{\tilde{Q}_{i}}f\|_{L^{p}(\si)}.
\end{split}
\end{equation*}
So finally it is enough to fix some $Q^{i}_{k_{i}}=:Q_{0}$, and assume that the function $f$ is supported on $Q_{0}$ and has zero $\si$-average.

We use a similar kind of splitting of the function inside the operator as in \cite{LL}, and a corresponding step appeared also in \cite{NTV2}. Consider some $Q \in \D$. Since the martingale differences $\de^{\si}_{Q}f$ have $\si$-integral zero, the term $\de_{Q}(f \si)$ in the square function can be written as
\begin{equation*}
\de_{Q}(f \si)=\de_{Q}\big((\de^{\si}_{Q}f+\sum_{R: R \supsetneq Q}\de^{\si}_{R}f)\si\big)
=\de_{Q}\big((\de^{\si}_{Q}f)\si\big)+\langle f \rangle^{\si}_{Q}\de_{Q}(1_{Q}\si).
\end{equation*}
Here we used that $f$ has zero average to get $\sum_{R: R \supsetneq Q}\de^{\si}_{R}f1_{Q}=\langle f \rangle^{\si}_{Q}1_{Q}$. Accordingly we split the estimate for the square function into two parts as
\begin{equation}\label{splitting of square}
\begin{split}
&\|S^{\si}_{b}(f)\|_{L^{q}(w)}
\leq \Big\| \Big(\sum_{Q \in \D}\Big( b_{Q}\de_{Q}\big((\de^{\si}_{Q}f) \si\big)\Big)^{2}\Big)^{\frac{1}{2}}\Big\|_{L^{q}(w)} \\
&+  \Big\| \Big(\sum_{Q \in \D}\Big( b_{Q}\langle f \rangle^{\si}_{Q}\de_{Q}(1_{Q}\si)\Big)^{2}\Big)^{\frac{1}{2}}\Big\|_{L^{q}(w)}.
\end{split}
\end{equation}

For the first term in the right hand side of (\ref{splitting of square}) we estimate 
\begin{equation*}
\big|\de_{Q}\big((\de^{\si}_{Q}f)\si\big)\big| 
\lesssim \frac{\int|\de^{\si}_{Q}f|\ud \sigma }{|Q|}1_{Q} 
=\langle |\de^{\si}_{Q}f| \rangle^{\si}_{Q} \frac{\si(Q)}{|Q|}1_{Q}.
\end{equation*}
This together with the $\mathscr{A}^{b}_{p,q}$-condition give
\begin{equation*}
\begin{split}
&\Big\| \Big(\sum_{Q \in \D}\Big( b_{Q}\de_{Q}\big((\de^{\si}_{Q}f) \si\big)\Big)^{2}\Big)^{\frac{1}{2}}\Big\|_{L^{q}(w)}
\lesssim \Big\| \Big(\sum_{Q \in \D} \Big(b_{Q}\frac{\int|\de^{\si}_{Q}f|\ud \sigma }{|Q|}\Big)^{2}1_{Q} \Big)^{\frac{1}{2}}\Big\|_{L^{q}(w)} \\
& \leq [\si,w]^{b}_{p,q} \Big\|\Big(\sum_{Q \in \D} \big(\langle |\de^{\si}_{Q}f| \rangle^{\si}_{Q}\big)^{2}1_{Q}\Big)^{\frac{1}{2}}\Big\|_{L^{p}(\si)} \\
&\lesssim  [\si,w]^{b}_{p,q} \Big\|\Big(\sum_{Q \in \D} \big(\de^{\si}_{Q}f\big)^{2}1_{Q}\Big)^{\frac{1}{2}}\Big\|_{L^{p}(\si)} 
\simeq [\si,w]^{b}_{p,q} \|f\|_{L^{p}(\si)},
\end{split}
\end{equation*}
where the second to last step follows from Stein's inequality (\ref{Stein}), and the last step follows from Burkholder's inequality (\ref{mart. norm}).

The last thing to do is to bound the second term in (\ref{splitting of square}). Let $\F$ be the collection of principal cubes for the function $f$ constructed beginning from the cube $Q_{0}$. 

Note that $\de_{Q}(1_{Q}\si)=\de_{Q}(1_{R}\si)$ for every cube $\D \ni R \supset Q$. Using the principal cubes we estimate
\begin{equation*}
\begin{split}
& \Big\| \Big(\sum_{Q \in \D}\Big( b_{Q}\langle f \rangle^{\si}_{Q}\de_{Q}(1_{Q}\si)\Big)^{2}\Big)^{\frac{1}{2}}\Big\|_{L^{q}(w)} \\
&\lesssim \Big\| \Big(\sum_{F\in \F}\big(\langle |f|\rangle^{\si}_{F}\big)^{2} \sum_{\begin{substack}{Q \in \D: \\ \pi_\F Q=F}\end{substack}} \big(b_{Q}\de_{Q} (1_{F}\si)\big)^{2}\Big)^{\frac{1}{2}}\Big\|_{L^{q}(w)} \\
&\leq\Big\| \Big(\sum_{F\in \F}\big(\langle |f|\rangle^{\si}_{F}\big)^{2} S^{\si}_{b,F}(1_{F})^{2}\Big)^{\frac{1}{2}}\Big\|_{L^{q}(w)} \\
&\leq \mathfrak{S}_{loc} \Big\|\Big(\sum_{F \in \F} \big(\langle |f|\rangle^{\si}_{F}\big)^{2}1_{F}\Big)^{\frac{1}{2}}\Big\|_{L^{p}(\sigma)} 
\lesssim \mathfrak{S}_{loc} \| f \|_{L^{p}(\si)},
\end{split}
\end{equation*}
where the last step follows from Carleson's embedding theorem (\ref{vector Carleson}). 

Note that we actually applied the quadratic testing condition only with a collection that is sparse with respect to the measure $\si$. This concludes the proof of Theorem \ref{two weight square}.
\end{proof}

\section{Dyadic shifts}\label{dyadic shifts}

Now we begin to consider the dyadic shifts. First we give some basic definitions and then we move on to characterize the two weight inequality.

For any interval $I \subset \R$ write $h^{0}_{I}:=|I|^{-\frac{1}{2}} 1_{I}$ and $h^{1}_{I}:= |I|^{-\frac{1}{2}} (1_{I_{l}}-1_{I_{r}})$, where $|I|$ is the length of the interval and $I_{l}$ and $I_{r}$ are the left and right halves of the interval, respectively. The function $h^{0}_{I}$ is called \emph{non cancellative}- and $h^{1}_{I}$ \emph{cancellative Haar function} related to the interval $I$. 

For a cube $Q=I_{1} \times I_{2} \times \dots \times I_{N} \in \D$, where each $I_{i}$ is an interval in $\R$, define for $ \eta \in \{0,1\}^{N}$ the Haar function related to the cube by 
\begin{equation*}
h^{\eta}_{Q}(x_{1}, \dots, x_{N}):= \prod_{i=1}^{N}h_{I_{i}}^{\eta_{i}}(x_{i}).
\end{equation*} If some $\eta_{i}$ is non zero, then $h^{\eta}_{Q}$ is called  cancellative since it has $\int h^{\eta}_{Q} \ud x =0$, and otherwise it is called non cancellative. In any case $ \int | h^{\eta}_{Q}|^{2} \ud x =1.$

Fix two non negative integers $m$ and $n$. For every cube $K\in \D$ suppose we have a linear operator $A^{\si}_{K}$ defined on locally $\si$-integrable functions by
\begin{equation}\label{building block}
A_{K}^{\si}f:= \sum_{\begin{substack}{I,J \in \D: \\ I^{(m)}=J^{(n)}= K}\end{substack}}a_{IJK} \langle f , h_{I}^J\rangle_{\si} h_{J}^I, 
\end{equation}
where $h_{I}^J$ is a Haar function related to the cube (not interval) $I \in \D$ and $h_{J}^I$ is a Haar function related to the cube $J \in \D$. The coefficients $a_{IJK} \in \R$ satisfy $|a_{IJK}| \leq \frac{\sqrt{|I| |J|}}{|K|}$. Here the Haar functions are just some Haar functions, not any specific ones, and hence we do not specify them with the superscript $\eta$. Similarly define the corresponding dual operator 
\begin{equation*}
A_{K}^{w}g:= \sum_{\begin{substack}{I,J \in \D: \\ I^{(m)}=J^{(n)}= K}\end{substack}}a_{IJK} \langle g , h_{J}^I\rangle_{w} h_{I}^J 
\end{equation*}
for locally $w$-integrable functions, where it should be noted that here the functions $h_{I}^J$ and $h_{J}^I$ are in ``opposite'' places. 

As a direct consequence of the size assumption of the coefficients we get for any $f \in L^{1}_{loc}(\si)$ that
\begin{equation}\label{average domination}
| A_{K}^{\si}f|  
\leq \frac{1}{|K|} \int_{K}|f| \ud \sigma 1_{K},
\end{equation}
and a similar estimate holds for $A_{K}^{w}$.

We assume that there are only finitely many $ K \in \D$ such that the coefficients $a_{IJK}$ are non zero. We make this assumption to have the dyadic shift well defined in the general two weight setting, but all the bounds below will be independent of this number.

With the operators $A_{K}^{\si}$ the dyadic shift $T^{\si}$ is defined by
\begin{equation}\label{shift}
T^{\si}f:=\sum_{K \in \D} A^{\si}_{K}f, \ \  \ \ f \in  L^{1}_{loc}(\si),
\end{equation}
and the shift $T^{w}$ is defined analogously with the operators $A^{w}_{K}$. They are formal adjoints of each other in the sense that
\begin{equation*}
\langle T^{\si}f, g \rangle_{w}=\langle f, T^{w}g \rangle_{\si}
\end{equation*}
for all $f \in L^{1}_{loc}(\si)$ and $g \in L^{1}_{loc}(w)$. The shift $T^{\si}$ is said to have \emph{parameters} $(m,n)$, and correspondingly the shift $T^{w}$ has parameters $(n,m)$. The number $\text{max}\{m,n\}$ is the \emph{complexity} of the shift.

Instead of a single dyadic shift we are going to consider a family $\mathscr{T}$ of dyadic shifts with at most a given complexity. Let us first recall the definition of $\mathcal{R}$-bounded operator families as used for example in \cite{W}. Suppose $(\varepsilon_{k})_{k=1}^{\infty}$ is a sequence of independent random signs. If $X$ and $Y$ are two Banach spaces and $\mathscr{T}$ is a family of linear operators from $X$ into $Y$, then $\mathscr{T}$ is said to be $\mathcal{R}$-bounded if there exists a constant $C$ such that for all $U\in \{1,2,\dots\}$, $(T_{u})_{u=1}^{U} \subset \mathscr{T}$ and $(x_{u})_{u=1}^{U}\subset X$ it holds that
\begin{equation}\label{R-boundedness}
\E\Big\| \sum_{u=1}^{U}\varepsilon_{u}T_{u}x_{u}\Big\|_{Y} 
\leq  C\E\Big\| \sum_{u=1}^{U}\varepsilon_{u}x_{u}\Big\|_{X}.
\end{equation}
We denote the smallest possible constant $C$ in \eqref{R-boundedness} by $\mathcal{R}(\mathscr{T})$.

If $X=L^{p}(\si)$ and $Y=L^{q}(w)$ for some $1\leq p,q<\infty$, then similar computations with the Kahane-Khinchine inequality as above with the dyadic square function shows that in this case $\mathcal{R}$-boundedness can be equivalently defined as
\begin{equation}\label{R-boundedness in L^{p}}
\Big\| \Big( \sum_{u=1}^{U} \big( T_{u} f_{u}\big)^{2} \Big)^{\frac{1}{2}} \Big\|_{L^{q}(w)}
\lesssim \mathcal{R}(\mathscr{T}) \Big\| \Big( \sum_{u=1}^{U}  f_{u}^{2} \Big)^{\frac{1}{2}} \Big\|_{L^{p}(\si)},
\end{equation}
where $\mathcal{R}(\mathscr{T})$ is the constant when formulated as in (\ref{R-boundedness}). If $p=q=2$ it is easily seen from (\ref{R-boundedness in L^{p}}) that in this case $\mathcal{R}$-boundedness is equivalent with uniform boundedness. On the other hand from (\ref{R-boundedness}) one sees that if $\mathscr{T}$ consists of  a single operator $T$, then $\mathcal{R}$-boundedness means just the boundedness of $T$.

Let $\mathscr{T}=\{T^{\si}_{\alpha}: \alpha \in \mathscr{A}\}$ be a collection of dyadic shifts. If $T^{\si}_{\alpha} \in \mathscr{T}$, then we write $T^{w}_{\alpha}$ for the corresponding formal adjoint.  We say that the collection $\mathscr{T}$ of dyadic shifts satisfies the (local) quadratic testing condition (with respect to exponents $1<p,q<\infty$) if for every $U \in \{1,2,\dots\}$, all sequences $(a_{u})_{u=1}^{U}\subset \R$, $(T^{\si}_{u})_{u=1}^{U}\subset \mathscr{T}$ and $(Q_{u})_{u=1}^{U} \subset \D$ the inequalities 
\begin{equation}\label{testing direct} 
\Big\| \Big( \sum_{u=1}^{U} \big(a_{u} 1_{Q_{u}}T_{u}^{\si} 1_{Q_{u}}\big)^{2} \Big)^{\frac{1}{2}} \Big\|_{L^{q}(w)}
\leq \mathcal{T}^{\si} \Big\| \Big( \sum_{u=1}^{U} a_{u}^{2}1_{Q_{u}} \Big)^{\frac{1}{2}} \Big\|_{L^{p}(\si)}
\end{equation}
and
\begin{equation}\label{testing dual} 
\Big\| \Big( \sum_{u=1}^{U} \big(a_{u} 1_{Q_{u}}T_{u}^{w} 1_{Q_{u}}\big)^{2} \Big)^{\frac{1}{2}} \Big\|_{L^{p'}(\si)}
\leq \mathcal{T}^{w} \Big\| \Big( \sum_{u=1}^{U} a_{u}^{2}1_{Q_{u}} \Big)^{\frac{1}{2}} \Big\|_{L^{q'}(w)}
\end{equation}
hold, where $\mathcal{T}^{\si}< \infty$ and $\mathcal{T}^{w}<\infty$ are the best possible constants. Note that it is not forbidden in this definition that $T_{u}=T_{u'}$ for some $u\not= u'$. In particular if $\mathscr{T}$ consists only of a single shift, then we get the corresponding quadratic testing condition as above with the dyadic square function.

The two weight theorem for the dyadic shifts is as follows:
\begin{theorem}\label{two weight shift}
Let $1<p,q<\infty$ be two exponents and assume that the measures $\si$ and $w$ satisfy the quadratic $\mathscr{A}_{p,q}$-condition. Suppose $\mathscr{T}$ is a collection of dyadic shifts as in (\ref{shift}) with complexities at most $\kappa$. Then the collection $\mathscr{T}$ is $\mathcal{R}$-bounded from $L^{p}(\si)$ into $L^{q}(w)$ if and only if it satisfies the quadratic testing conditions (\ref{testing direct}) and (\ref{testing dual}), and in this case 
\begin{equation}\label{R-bound}
\mathcal{R}(\mathscr{T}) \lesssim (1+\kappa)(\mathcal{T}^{\si} + \mathcal{T}^{w}) + (1+\kappa)^{2}[\si,w]_{p,q}.
\end{equation}

\end{theorem}

Again before proving the theorem we comment quickly on the case $1<p\leq 2\leq q<\infty$. Similar computations as in (\ref{simple strong}) show that in this case $\mathcal{R}$-boundedness is equivalent to uniform boundedness, the quadratic testing condition reduces to Sawyer type testing and the quadratic $\mathscr{A}_{p,q}$-condition becomes the simple $A_{p,q}$-condition. Thus we get that a dyadic shift $T^{\si}$ is bounded from $L^{p}(\si)$ into $L^{q}(w)$ if and only if the Sawyer type conditions
\begin{equation*}
\|1_{Q}T^{\si}1_{Q} \|_{L^{q}(w)} \leq \mathcal{T}^{\si} \si(Q)^{\frac{1}{p}} 
\end{equation*}
and 
\begin{equation*}
\|1_{Q}T^{w}1_{Q} \|_{L^{p'}(\si)} \leq \mathcal{T}^{w} w(Q)^{\frac{1}{q'}} 
\end{equation*}
hold for all $ Q \in \D$, and the measures satisfy the Muckenhoupt type $A_{p,q}$-condition
\begin{equation*}
(\si,w)_{p,q}:=\sup_{Q \in \D} \frac{\si(Q)^{\frac{1}{p'}}w(Q)^{\frac{1}{q}}}{|Q|}<\infty.
\end{equation*}  
In this case 
\begin{equation*}
\|T^{\si}\|_{L^{p}(\si) \to L^{q}(w)} \lesssim (1+\kappa)(\mathcal{T}^{\si} + \mathcal{T}^{w}) + (1+\kappa)^{2}(\si,w)_{p,q},
\end{equation*}
which is the result proved in \cite{HPTV} when $p=q=2$.

\begin{proof}[Proof of Theorem \ref{two weight shift}]
Suppose $\mathscr{T}$ is $\mathcal{R}$-bounded,  whence clearly the quadratic testing condition (\ref{testing direct}) is satisfied. Using duality one sees that the collection of formal adjoints of the shifts in $\mathscr{T}$ is $\mathcal{R}$-bounded from $L^{q'}(w)$ into $L^{p'}(\si)$, and thus also (\ref{testing dual}) is satisfied. Hence it is enough to show the sufficiency of the testing conditions.

So we assume that we have a collection $\mathscr{T}$ of dyadic shifts with complexity at most $\kappa$ satisfying the quadratic testing conditions (\ref{testing direct}) and (\ref{testing dual}). For any $U=1,2,\dots$ suppose we have some sequences $(T^{\si}_{u})_{u=1}^{U}\subset \mathscr{T}$ and $(f_{u})_{u=1}^{U}\subset L^{p}(\si)$.  To prove (\ref{R-bound}) it is enough to take an arbitrary sequence $(g_{u})_{u=1}^{U}\subset L^{q'}(w)$ and show that
\begin{equation*}
\begin{split}
&\Big|\sum_{u=1}^{U} \langle T^{\si}_{u}f_{u},g_{u}\rangle_{w}\Big| \\
&\lesssim  \big((1+\kappa)(\mathcal{T}^{\si} + \mathcal{T}^{w}) + (1+\kappa)^{2}(\si,w)_{p,q}\big) \|(f_{u})_{u=1}^{U}\|_{L^{p}(\si;l^{2})}\|(g_{u})_{u=1}^{U}\|_{L^{q'}(w;l^{2})}.
\end{split}
\end{equation*}
For every $u$ we write the corresponding shift as
\begin{equation*}
T^{\si}_{u}f_{u}=\sum_{K \in \D} A^{\si}_{u,K}f_{u} 
= \sum_{K \in \D} \sum_{\begin{substack}{I,J \in \D: \\ I^{(m)}=J^{(n)}= K}\end{substack}}a^{u}_{IJK} \langle f_{u} , h_{I,u}^J\rangle_{\si} h_{J,u}^I.
\end{equation*}

Let again $\cup_{k=1}^{\infty} Q^{i}_{k}$, $i=1,\dots,j\leq 2^{N}$, be the different ``quadrants''  of our dyadic system, as explained
around equation (\ref{quadrants}). Because we assumed that every shift consists of only finitely many operators $A^{\si}_{K}$, we can choose for every $i$ a cube $Q^{i}_{k_{i}}:=\tilde{Q}_{i}$ such that $a^{u}_{IJK}\not=0$ implies  $K \subset \cup_{i=1}^{j} \tilde{Q}_{i}$.
Since the definition of the shift shows that $T^{\si}_{u}(f_{u}1_{\tilde{Q}_{i}})$ is supported on $1_{\tilde{Q}_{i}}$, we have
\begin{equation*}
\sum_{u=1}^{U}\big\langle T^{\si}_{u}f_{u},g_{u}\big\rangle_{w}
=\sum_{i=1}^{j}\sum_{u=1}^{U}\big\langle T^{\si}_u1_{\tilde{Q}_{i}}f_{u},1_{\tilde{Q}_{i}}g_{u}\big\rangle_{w},
\end{equation*}
and it is enough to estimate for each $i$ separately.

Finally we split 
\begin{equation}\label{to zero average}
\begin{split}
&\big\langle T^{\si}_{u}1_{\tilde{Q}_{i}}f_{u},1_{\tilde{Q}_{i}}g_{u}\big\rangle_{w} 
=\Big\langle T^{\si}_{u}\big(1_{\tilde{Q}_{i}}(f_{u}-\langle f_{u} \rangle^{\si}_{\tilde{Q}_{i}})\big),1_{\tilde{Q}_{i}}(g_{u}-\langle g_{u} \rangle^{w}_{\tilde{Q}_{i}})\Big\rangle_{w} \\
&+\Big\langle 1_{\tilde{Q}_{i}}(f_{u}-\langle f_{u} \rangle^{\si}_{\tilde{Q}_{i}}),\langle g_{u} \rangle^{w}_{\tilde{Q}_{i}}T^{w}_{u}1_{\tilde{Q}_{i}}\Big\rangle_{\si}
+ \Big\langle \langle f_{u} \rangle^{\si}_{\tilde{Q}_{i}}T^{\si}1_{\tilde{Q}_{i}},1_{\tilde{Q}_{i}}g_{u}\Big\rangle_{w},
\end{split}
\end{equation}
and the sum over $u$ of the last two terms can be bounded directly with the testing conditions. For example 
\begin{equation*}
\begin{split}
&\Big|\sum_{u=1}^{U} \langle f_{u} \rangle^{\si}_{\tilde{Q}_{i}}\big\langle T^{\si}_{u}1_{\tilde{Q}_{i}},1_{\tilde{Q}_{i}}g_{u}\big\rangle_{w} \Big| \\
& \leq \Big\| \Big( \sum_{u=1}^{ U} \big( \langle f_{u} \rangle^{\si}_{\tilde{Q}_{i}}1_{\tilde{Q}_{i}} T^{\si}_{u}1_{\tilde{Q}_{i}}\big)^{2} \Big)^{\frac{1}{2}} \Big\|_{L^{q}(w)}
\Big\| \Big( \sum_{u=1}^{ U} | 1_{\tilde{Q}_{i}}g_{u}|^{2} \Big)^{\frac{1}{2}} \Big\|_{L^{q'}(w)}\\
& \leq \mathcal{T}^{\si} \Big(\sum_{u=1}^{ U} \big( \langle f_{u} \rangle^{\si}_{\tilde{Q}_{i}} \big)^{2}\Big)^{\frac{1}{2}} \si(\tilde{Q}_{i})^{\frac{1}{p}} 
\|(1_{\tilde{Q}_{i}}g_{u})_{u=1}^{U}\|_{L^{q'}(w;l^{2})},
\end{split}
\end{equation*}
and using the fact that an $l^{2}$-sum of averages is less than the average of the $l^{2}$-sum we get

\begin{equation*}
\Big(\sum_{u=1}^{ U} \big( \langle f_{u} \rangle^{\si}_{\tilde{Q}_{i}} \big)^{2}\Big)^{\frac{1}{2}} \si(\tilde{Q}_{i})^{\frac{1}{p}} 
\leq \Big\langle\Big(\sum_{u=1}^{ U} f_{u}^{2}\Big)^{\frac{1}{2}}\Big\rangle^{\si}_{\tilde{Q_{i}}} \si(\tilde{Q}_{i})^{\frac{1}{p}} 
\leq \|(1_{\tilde{Q}_{i}}f_{u})_{u=1}^{U}\|_{L^{p}(\si;l^{2})}.
\end{equation*}

Thus after these reductions it is enough to fix one cube $Q^{i}_{k_{i}}=:Q_{0}$ and suppose that for every $u$ the functions $f_{u}$ and $g_{u}$ are supported on $Q_{0}$ and have zero averages. Since the shifts $T^{\si}_{u}$ are \emph{a priori} bounded, we can by $L^{p}$-convergence of martingale differences assume that the functions are given by
\begin{equation*}
f_{u}=\sum_{\begin{substack}{Q \in \D: \\ Q \subset Q_{0}}\end{substack}} \de^{\si}_{Q}f_{u}, \ \ 
g_{u}=\sum_{\begin{substack}{Q \in \D: \\ Q \subset Q_{0}}\end{substack}} \de^{w}_{Q}g_{u},
\end{equation*}
where the sums are finite.

Using the martingale decomposition

\begin{equation}\label{split shift}
\sum_{u=1}^{U} \langle T^{\si}_{u}f_{u},g_{u}\rangle_{w}
=\sum_{u=1}^{U}\sum_{Q,R\in \D} \langle T^{\si}_{u}\de^{\si}_{Q}f_{u},\de^{w}_{R}g_{u}\rangle_{w}, 
\end{equation}
we split the the proof into parts depending on the relative positions of the cubes $Q$ and $R$, and this part of the proof follows the outlines in \cite{H2}. The cases ``$l(Q) \leq l(R)$'' and ``$l(Q)>l(R)$'' are treated symmetrically, and here we concentrate on the first. Then, using the maximal possible complexity $\kappa$ of the shifts, we further split into three cases 	`` $Q \cap R =\emptyset$'', ``$Q^{(\kappa)} \subsetneq R$'' and ``$ Q \subset R \subset Q^{(\kappa)}$'', and these are treated separately using different properties of the shifts.

In the summations we understand that we are summing over dyadic cubes, and  we will not always write ``$Q \in \D$'' in the summation condition. Moreover, since we assumed the finite martingale decompositions of $f$ and $g$, we can think that every $Q \in \D$ that appears below will actually belong to some sufficiently big \emph{finite} collection $\D_{0} \subset \D$. This way all the sums are actually finite, and one does not have to worry about any convergence issues.

At this point it is convenient to introduce the notation 
\begin{equation*}
\de^{\si,i}_{Q}f:= \sum_{\begin{substack}{Q'\in \D: \\ Q'^{(i)}=Q}\end{substack}} \de^{\si}_{Q'}f
\end{equation*}
for any $f \in L^{1}_{loc}(\si)$, $Q\in \D$ and $i \in \{ 0,1,2,\dots\}$, and similarly for the measure $w$.

\subsection*{Disjoint cubes; $Q\cap R=\emptyset$ and $l(Q) \leq l(R)$}
Here we bound the part
\begin{equation}\label{case 1}
\Big|\sum_{u=1}^{U}\sum_{\begin{substack}{l(Q) \leq l(R) \\ Q \cap R=\emptyset}\end{substack}} \langle T^{\si}_{u}\de^{\si}_{Q}f_{u},\de^{w}_{R}g_{u}\rangle_{w}\Big|.
\end{equation}
Consider a fixed $u$ first, and suppose the shift $T^{\si}_{u}$ has parameters $(m,n)$ with $m+n\leq \kappa$. Fix two cubes $Q, R \in \D$  with $Q\cap R =\emptyset$ and suppose $ K \in \D$ is such that $\langle A^{\si}_{u,K}\de^{\si}_{Q}f_{u},\de^{w}_{R}g_{u}\rangle_{w}\not=0$.  We must have $Q \cap K \not= \emptyset \not= R \cap K$, which combined with $Q \cap R = \emptyset$ implies that $Q,R \subset K$. Also, since the functions $\de^\si_Qf $ and $\de^w_Rg$ have zero $\si$- and $w$-averages, respectively, and a Haar function $h_I$ is constant on the children of $I$,  we have  $K \subset Q^{(m)}$ and $ K \subset R^{(n)}$. Thus the sum (\ref{case 1}) is actually zero if $m=0$ or $n=0$. Hence we assume $m,n \geq 1$, rearrange the sum in question and estimate with (\ref{average domination}) as
\begin{equation}\label{first appl. of Ap}
\begin{split}
&\sum_{\begin{substack}{l(Q) \leq l(R) \\ Q \cap R=\emptyset}\end{substack}} \big| \big\langle T^{\si}_{u}\de^{\si}_{Q}f_{u},\de^{w}_{R}g_{u}\big\rangle_{w}\big| \\
&\leq \sum_{i=1}^{m}\sum_{j=1}^{n} \sum_{K \in \D} \sum_{\begin{substack}{Q,R \in \D : \\Q^{(i)}=R^{(j)}=K}\end{substack}} 
 \big| \big\langle A_{u,K}^{\si}\de^{\si}_{Q}f_{u}, \de^{w}_{R}g_{u} \big\rangle_{w}\big| \\
& \leq \sum_{i,j=1}^{\kappa} \sum_{K \in \D} \sum_{\begin{substack}{Q,R \in \D : \\Q^{(i)}=R^{(j)}=K}\end{substack}} 
\frac{\|\de^{\si}_{Q}f_{u}\|_{L^{1}(\si)}\|\de^{w}_{R}g_{u}\|_{L^{1}(w)}}{|K|} \\
&=\sum_{i,j=1}^{\kappa} \sum_{K \in \D} \frac{\|\de^{\si,i}_{K}f_{u}\|_{L^{1}(\si)}\|\de^{w,j}_{K}g_{u}\|_{L^{1}(w)}}{|K|}.
\end{split}
\end{equation}
Note that this estimate does not depend on the parameters $(m,n)$ of the shift.

Then for any fixed $i$ and $j$, we sum over $u$, and continue with 
\begin{equation}\label{towards Ap}
\begin{split}
&\sum_{u=1}^{U}\sum_{K \in \D} \frac{\|\de^{\si,i}_{K}f_{u}\|_{L^{1}(\si)}\|\de^{w,j}_{K}g_{u}\|_{L^{1}(w)}}{|K|} \\
&=\int \sum_{u=1}^{U}\sum_{K \in \D} \frac{\|\de^{\si,i}_{K}f_{u}\|_{L^{1}(\si)}}{|K|}|\de^{w,j}_{K}g_{u}| \ud w \\
& \leq \Big\|\Big(\sum_{u=1}^{U}\sum_{K \in \D} \Big(\frac{\|\de^{\si,i}_{K}f_{u}\|_{L^{1}(\si)}}{|K|}\Big)^{2}1_{K}\Big)^{\frac{1}{2}}\Big\|_{L^{q}(w)} \\
&  \cdot \Big\|\Big(\sum_{u=1}^{U}\sum_{K \in \D} \big(\de^{w,j}_{K}g_{u}\big)^{2}1_{K}\Big)^{\frac{1}{2}}\Big\|_{L^{q'}(w)}=:A\cdot B.
\end{split}
\end{equation}
 
Using the quadratic $\mathscr{A}_{p,q}$-condition we get
\begin{equation}\label{application of Ap}
\begin{split}
A=&\Big\|\Big(\sum_{u=1}^{U}\sum_{K \in \D} \Big(\big\langle |\de^{\si,i}_{K}f_{u}|\big\rangle^{\si}_{K}\frac{\si(K)}{|K|}\Big)^{2}1_{K}\Big)^{\frac{1}{2}}\Big\|_{L^{q}(w)} \\
&\leq [\si,w]_{p,q} \Big\|\Big(\sum_{u=1}^{U}\sum_{K \in \D} \big(\big\langle |\de^{\si,i}_{K}f_{u}|\big\rangle^{\si}_{K}\big)^{2}1_{K}\Big)^{\frac{1}{2}}\Big\|_{L^{p}(\si)} \\
& \leq [\si,w]_{p,q} \Big\|\Big(\sum_{K \in \D} \Big(\Big\langle\big(\sum_{u=1}^{U} \big(\de^{\si,i}_{K}f_{u}\big)^{2}\big)^{\frac{1}{2}} \Big\rangle^{\si}_{K}\Big)^{2}1_{K}\Big)^{\frac{1}{2}}\Big\|_{L^{p}(\si)}.
\end{split}
\end{equation}
Applying Stein's inequality (\ref{Stein}) and then Burkholder's inequality (\ref{Burkholder}) to the last term in (\ref{application of Ap}) we have
\begin{equation*}
\begin{split}
&RHS(\ref{application of Ap}) \lesssim [\si,w]_{p,q} \Big\|\Big(\sum_{K \in \D} \sum_{u=1}^{U} \big(\de^{\si,i}_{K}f_{u}\big)^{2} 1_{K}\Big)^{\frac{1}{2}}\Big\|_{L^{p}(\si)} \\
& \lesssim [\si,w]_{p,q} \|(f_{u})_{u=1}^{U}\|_{L^{p}(\si;l^{2})}.
\end{split}
\end{equation*}

The factor $B$ in (\ref{towards Ap}) is estimated directly with Burkholder's inequality, and then it only remains to sum over the finite ranges of $i$ and $j$, which produces a factor $\kappa^{2}$ in the final estimate. Hence we have shown that
\begin{equation*}
(\ref{case 1}) \lesssim \kappa^{2} \cdot [\si,w]_{p,q} \|(f_{u})_{u=1}^{U}\|_{L^{p}(\si;l^{2})}\|(g_{u})_{u=1}^{U}\|_{L^{q'}(w;l^{2})}.
\end{equation*}

\subsection*{Deeply contained cubes; $Q^{(\kappa)} \subsetneq R$}
We consider again a fixed $T^{\si}_{u}$ with parameters $(m,n)$ first.  Assume $Q,R \in \D$ are two cubes such that $Q^{(\kappa)} \subsetneq R$. If $ A^{\si}_{u,K}\de^{\si}_{Q}f_{u}$ is  non zero, we must have $K \subset Q^{(m)} \subset Q^{(\kappa)} \subsetneq R $. Since $A^{\si}_{u,K}\de^{\si}_{Q}f_{u}$ is supported on $K$ and $\de^w_Rg$ is constant on the children of $R$, we see that 
\begin{equation*}
\Big\langle A^{\si}_{u,K}\de^{\si}_{Q}f_{u},\de^{w}_{R}g_{u}\Big\rangle_{w}
=\Big\langle A^{\si}_{u,K}\de^{\si}_{Q}f_{u},\langle\de^{w}_{R}g_{u}\rangle^{w}_{Q^{(\kappa)}}1_{Q^{(\kappa)}}\Big\rangle_{w},
\end{equation*}
and thus
\begin{equation*}
\Big\langle T_{u}^{\si}\de^{\si}_{Q}f_{u},\de^{w}_{R}g_{u}\Big\rangle_{w}
=\Big\langle T_{u}^{\si}\de^{\si}_{Q}f_{u},\langle \de^{w}_{R}g_{u}\rangle^{w}_{Q^{(\kappa)}}1_{Q^{(\kappa)}}\Big\rangle_{w}.
\end{equation*}

Taking ``$Q^{(\kappa)}$'' as a new summation variable we can rewrite the sum to be estimated as
\begin{equation}\label{organize deeply contained}
\begin{split}
&\sum_{\begin{substack}{Q,R \in \D : \\ Q^{(\kappa)}\subsetneq R}\end{substack}} \Big\langle T_{u}^{\si}\de^{\si}_{Q}f_{u},\de^{w}_{R}g_{u}\Big\rangle_{w}
 = \sum_{Q \in \D} \sum_{\begin{substack}{R \in \D : \\R \supsetneq Q}\end{substack}} \sum_{\begin{substack}{Q' \in \D : \\ Q'^{(\kappa)} =Q}\end{substack}}
\Big\langle T_{u}^{\si}\de^{\si}_{Q'}f_{u},\langle \de^{w}_{R}g_{u}\rangle^{w}_{Q}1_{Q}\Big\rangle_{w} \\
& = \sum_{Q \in \D} \Big\langle \de^{\si,\kappa}_{Q}f_{u}, \langle g_{u} \rangle^{w}_{Q} \de^{\si,\kappa}_{Q} T_{u}^{w}1_{Q}\Big\rangle_{\si},
\end{split}
\end{equation}
where we collapsed the sum $\sum_{\begin{substack}{R \in \D : \\R \supsetneq Q}\end{substack}} \langle \de^{w}_{R}g_{u}\rangle^{w}_{Q}1_{Q} = \langle g_{u} \rangle^{w}_{Q} 1_{Q}$, and used the fact that the martingale difference operator $\de^{\si,\kappa}_{Q}$ can be put also to the other side of the pairing $\langle \cdot, \cdot\rangle_{\si}$. Now we have again an equation that is independent of the parameters $(m,n)$, so it holds for all the shifts $T^{\si}_{u}$.

Then we sum over $u$ and estimate up as
\begin{equation}\label{deeply up}
\begin{split}
&\Big|\sum_{u=1}^{U} \sum_{Q \in \D} \Big\langle \de^{\si,\kappa}_{Q}f_{u}, \langle g_{u} \rangle^{w}_{Q} \de^{\si,\kappa}_{Q} T_{u}^{w}1_{Q}\Big\rangle_{\si} \Big| \\
&= \Big | \int \sum_{u=1}^{U} \sum_{Q \in \D}  \de^{\si,\kappa}_{Q}f_{u} \langle g_{u} \rangle^{w}_{Q} \de^{\si,\kappa}_{Q} T_{u}^{w}1_{Q}  \ud\si \Big| \\
&\leq \Big\| \Big(\sum_{u=1}^{U} \sum_{Q \in \D} \big( \de^{\si,\kappa}_{Q}f_{u}  \big)^{2} \Big)^{\frac{1}{2}} \Big\|_{L^{p}(\si)} \\
&\cdot \Big\| \Big( \sum_{u=1}^{U} \sum_{Q \in \D} \big( \langle g_{u} \rangle^{w}_{Q} \de^{\si,\kappa}_{Q} T^{w}_{u}1_{Q}\big)^{2} \Big)^{\frac{1}{2}} \Big\|_{L^{p'}(\si)}, 
\end{split}
\end{equation}
where Burkholder's inequality (\ref{Burkholder}) implies that the first factor is dominated by $\|(f_{u})_{u=1}^{U}\|_{L^{p}(\si;l^{2})}$. 

In the second factor we note that if $\varphi$ is any locally $w$-integrable function, then $\de^{\si,\kappa}_{Q} A_{u,K}^{w} (1_{\complement Q}\varphi) =0$ for any $ Q, K \in \D$, which follows from the fact that the shift has complexity at most $\kappa$. This shows that
\begin{equation}\label{replacing with bigger}
\de^{\si,\kappa}_{Q} T_{u}^{w}1_{Q} = \de^{\si,\kappa}_{Q} T_{u}^{w}1_{P}
\end{equation}
for any $\D \ni P \supset Q$. 

Beginning from the cube $Q_{0}$, construct the sets $\G_{u}$ of principal cubes for the functions $g_{u}$ with respect to the measure $w$. Since the functions $g_{u}$ have finite martingale difference decompositions, and are accordingly constant on sufficiently small cubes $Q \in \D$, the collections $\G_{u}$ are finite. 

With the remark (\ref{replacing with bigger}) we proceed with 
\begin{equation*}
\begin{split}
&\Big\| \Big( \sum_{u=1}^{U} \sum_{Q \in \D} \big( \langle g_{u} \rangle^{w}_{Q} \de^{\si,\kappa}_{Q} T^{w}_{u}1_{Q}\big)^{2} \Big)^{\frac{1}{2}} \Big\|_{L^{p'}(\si)} \\ 
& \lesssim \Big\| \Big( \sum_{u=1}^{U} \sum_{G \in \G_{u}} \big(\langle |g_{u}| \rangle^{w}_{G}\big)^{2} \sum_{\begin{substack}{Q \in \D: \\ \pi_{\G_u} Q = G}\end{substack}} \big(  \de^{\si,\kappa}_{Q} T_{u}^{w}1_{G}\big)^{2} \Big)^{\frac{1}{2}} \Big\|_{L^{p'}(\si)} \\ 
& \lesssim \Big\| \Big( \sum_{u=1}^{U} \sum_{G \in \G_{u}} \big(\langle |g_{u}| \rangle^{w}_{G}1_G T_{u}^{w}1_{G}\big)^{2} \Big)^{\frac{1}{2}} \Big\|_{L^{p'}(\si)} \\
& \leq \mathcal{T}^{w} \Big\| \Big( \sum_{u=1}^{U}\sum_{G\in \G_{u} } \big(\langle | g_{u} | \rangle_{G}^{w} 1_{G}\big)^{2} \Big)^{\frac{1}{2}} \Big\|_{L^{q'}(w)}  
\lesssim \mathcal{T}^{w} \| (g_{u})_{u=1}^{U} \|_{L^{q'}(w;l^{2})},
\end{split}
\end{equation*}
where we used Burkholder's inequality \eqref{Burkholder} in the second step and  Carleson's embedding theorem (\ref{vector Carleson}) in the last step. This concludes the proof for the part ``$Q^{(\kappa)} \subsetneq R$''.

\subsection*{Contained cubes of comparable size; $ Q \subset R \subset Q^{(\kappa)}$} For a fixed $u$, the sum to be estimated in this last subsection  can be written as
\begin{equation}\label{comparable splitting}
\begin{split}
&\sum_{i=0}^{\kappa} \sum_{R \in \D} \sum_{\begin{substack}{Q \in \D: \\ Q^{(i)} = R}\end{substack}} \big\langle T^{\si}_{u}\de^{\si}_{Q}f_{u},\de^{w}_{R}g_{u}\big\rangle_{w} \\
&=\sum_{i=0}^{\kappa} \sum_{k=1}^{2^{N}} \sum_{R \in \D} \Big\langle \de^{\si,i}_{R}f_{u}, \langle \de^{w}_{R}g_{u} \rangle^{w}_{R_{k}}T^{w}_{u}1_{R_{k}}\Big\rangle_{\si} \\
&=\sum_{i=0}^{\kappa} \sum_{k=1}^{2^{N}} \sum_{R \in \D}  \Big\langle  1_{R_{k}}\de^{\si,i}_{R}f_{u} , \langle \de^{w}_{R}g_{u} \rangle^{w}_{R_{k}}T^{w}_{u}1_{R_{k}}\Big\rangle_{\si} \\
&+\sum_{i=0}^{\kappa} \sum_{k=1}^{2^{N}} \sum_{R \in \D} \Big\langle 1_{\complement R_{k}} \de^{\si,i}_{R}f_{u} , \langle \de^{w}_{R}g_{u} \rangle^{w}_{R_{k}}T^{w}_{u}1_{R_{k}}\Big\rangle_{\si},
\end{split}
\end{equation}
where the cubes $R_{k}$ are the dyadic children of $R$.

Consider the  first sum in the right side of (\ref{comparable splitting}). We fix some $i$ and $k$, sum over $u$ and use testing to deduce that
\begin{equation*}
\begin{split}
& \Big| \sum_{u=1}^{U} \sum_{R \in \D}  \Big\langle 1_{R_{k}} \de^{\si,i}_{R}f_{u} , \langle \de^{w}_{R}g_{u} \rangle^{w}_{R_{k}}T^{w}_{u}1_{R_{k}}\Big\rangle_{\si} \Big| \\
&\leq  \Big\| \Big( \sum_{u=1}^{U} \sum_{R \in \D} \big(1_{R_{k}} \de^{\si,i}_{R}f_{u}  \big)^{2} \Big)^{\frac{1}{2}} \Big\|_{L^{p}(\si)} \\
&\cdot \Big\| \Big(\sum_{u=1}^{U} \sum_{R \in \D} \big( \langle \de^{w}_{R}g_{u} \rangle^{w}_{R_{k}}1_{R_{k}}T^{w}_{u}1_{R_{k}}\big)^{2} \Big)^{\frac{1}{2}} \Big\|_{L^{p'}(\si)} \\
&\lesssim  \mathcal{T}^{w} \|(f_{u})_{u=1}^{U} \|_{L^{p}(\si;l^{2})} 
\Big\| \Big(\sum_{u=1}^{U} \sum_{R \in \D} | \langle \de^{w}_{R}g \rangle^{w}_{R_{k}}1_{R_{k}}|^{2} \Big)^{\frac{1}{2}} \Big\|_{L^{q'}(w)} \\
& \lesssim \mathcal{T}^{w} \|(f_{u})_{u=1}^{U} \|_{L^{p}(\si;l^{2})} \|(g_{u})_{u=1}^{U} \|_{L^{q'}(w;l^{2})}. 
\end{split}
\end{equation*}

Now turn to the other sum in (\ref{comparable splitting}) to be estimated. With the same notation as there, we have $ 1_{\complement R_{k}} A_{u,K}^{w}1_{R_{k}} \not=0$ only if $ K \supset R$. Hence, using (\ref{average domination}), we get
\begin{equation*}
\begin{split}
&\Big|\Big\langle 1_{\complement R_{k}} \de^{\si,i}_{R}f_{u} , \langle \de^{w}_{R}g_{u} \rangle^{w}_{R_{k}}T^{w}_{u}1_{R_{k}}\Big\rangle_{\si}\Big| 
\leq \sum_{\begin{substack}{K \in \D: \\ K \supset  R}\end{substack}} \frac{ \|1_{\complement R_{k}} \de^{\si,i}_{R}f_{u}  \|_{L^{1}(\si)} 
\| 1_{R_{k}} \de^{w}_{R}g_{u} \|_{L^{1}(w)}  }{|K|} \\
& \simeq \frac{ \|1_{\complement R_{k}} \de^{\si,i}_{R}f_{u}  \|_{L^{1}(\si)} 
\|1_{R_{k}} \de^{w}_{R}g_{u} \|_{L^{1}(w)}  }{|R|}.
\end{split}
\end{equation*}

Summing this over $k$, and then over $R \in \D$ and $u\in \{1, \dots,U\}$ leads, similarly as in equations (\ref{towards Ap}) and (\ref{application of Ap}), to
\begin{equation}\label{secon appl. of Ap}
\begin{split}
& \sum_{u=1}^{U} \sum_{R \in \D} \sum_{k=1}^{2^{N}} \frac{ \| 1_{\complement R_{k}}\de^{\si,i}_{R}f_{u}  \|_{L^{1}(\si)} 
\|1_{R_{k}} \de^{w}_{R}g_{u} \|_{L^{1}(w)}  }{|R|} \\
\leq  &\Big\|\Big(\sum_{u=1}^{U}\sum_{R \in \D} \Big(   \frac{\|\de^{\si,i}_{R} f_{u}\|_{L^{1}(\si)}}{|R|}\Big)	^{2}1_{R}\Big)^{\frac{1}{2}}\Big\|_{L^{q}(w)}\\
\cdot& \Big\|\Big(\sum_{u=1}^{U}\sum_{R \in \D} \big(\de^{w}_{R}g_{u} \big)^{2}\Big)^{\frac{1}{2}}\Big\|_{L^{q'}(w)} \\
 \lesssim &[\si,w]_{p,q} \|(f_{u})_{u=1}^{U} \|_{L^{p}(\si;l^{2})} \|(g_{u})_{u=1}^{U} \|_{L^{q'}(w;l^{2})}.
\end{split}
\end{equation}
Summing over $i\in \{0,\dots, \kappa\}$ produces the factor $1+\kappa$ in the final estimate. 

This finishes the proof of the case ``$ Q \subset R \subset Q^{(\kappa)}$'', and hence also of Theorem \ref{two weight shift}.

\end{proof}

\begin{lem}
Let $1<p,q<\infty$ and suppose $\mathscr{T}$ is a family of dyadic shifts containing all shifts with parameters $(m,n)$. If $\mathscr{T}$ is $\mathcal{R}$-bounded from $L^{p}(\si)$ into $L^{q}(w)$, then
\begin{equation*}
[\si,w]_{p,q} \leq 2^{N\min(m,n)} \mathcal{R}(\mathscr{T}).
\end{equation*} 
\end{lem}

\begin{proof}
Suppose for example that $m\leq n$. The situation $m>n$ is similar. For every $I \in \D$ define the shift
\begin{equation*}
T^{\si}_{I}:= \sum_{\begin{substack}{J \in \D: \\ J^{(n-m)}=I}\end{substack}} \frac{\sqrt{|I||J|}}{|I^{(m)}|} \langle \cdot, h_{I}\rangle_{\si}h_{J},
\end{equation*}
where the functions $h_{I}$ and $h_{J}$ are some fixed Haar functions related to the cubes $I$ and $J$. Define also the function $f_{I}:= h_{I} \sqrt{|I|}$.

With these definitions we have $|T^{\si}_{I}f_{I}| = \frac{\si(I)}{2^{Nm}|I|}1_{I}$, and clearly $|f_{I}|=1_{I}$. Thus, if $\{a_{I}\}_{I \in \D}$ is any finitely non zero set of real numbers, then
\begin{equation*}
\begin{split}
&2^{-Nm} \Big\| \Big( \sum_{I \in \D} \Big(a_{I} \frac{\si(I)}{|I|}1_{I}\Big)^{2} \Big)^{\frac{1}{2}} \Big\|_{L^{q}(w)}
=\Big\| \Big( \sum_{I \in \D} \big(a_{I} T^{\si}_{I} f_{I}\big)^{2} \Big)^{\frac{1}{2}} \Big\|_{L^{q}(w)} \\
& \leq  \mathcal{R}(\mathscr{T}) \Big\| \Big( \sum_{I \in \D} \big(a_{I} f_{I} 1_{I}\big)^{2} \Big)^{\frac{1}{2}} \Big\|_{L^{p}(\si)}
=\mathcal{R}(\mathscr{T})\Big\| \Big( \sum_{I \in \D} a_{I}^{2}1_{I} \Big)^{\frac{1}{2}} \Big\|_{L^{p}(\si)},
\end{split}
\end{equation*}
which shows that 
$[\si,w]_{p,q} \leq 2^{Nm} \mathcal{R}(\mathscr{T})$.
\end{proof}

\begin{cor}
Suppose $1<p,q<\infty$. The family $\mathscr{T}$ of all shifts with parameters $(m,n)$ is $\mathcal{R}$-bounded from $L^{p}(\si)$ into $L^{q}(w)$ if and only if the family satisfies the quadratic testing conditions (\ref{testing direct}) and (\ref{testing dual}), and the quadratic $\mathscr{A}_{p,q}$-condition holds. Moreover we have the quantitative estimate
\begin{equation*}
2^{-N\min(m,n)}[\si,w]_{p,q} +\mathcal{T}^{\si} + \mathcal{T}^{w} \lesssim  \mathcal{R}(\mathscr{T})
\lesssim (1+\kappa)(\mathcal{T}^{\si} + \mathcal{T}^{w}) + (1+\kappa)^{2}[\si,w]_{p,q},
\end{equation*}
where $\mathcal{T}^{\si}$ and $\mathcal{T}^{w}$ are the testing constants and $\kappa=\max\{m,n\}$.
\end{cor}

\subsection*{Dyadic shifts of a specific form}
We look at the case when all the operators $A^{\si}_{K}$ in the definition of the dyadic shifts are of the form
\begin{equation}\label{special case}
A_{K}^{\si}f:= \sum_{\begin{substack}{I,J : I^{(m)}=J^{(n)}= K \\ I \vee J=K}\end{substack}}a_{IJK} \langle f , h_{I}^J\rangle_{\si} h_{J}^I,
\end{equation}
where $I\vee J$ denotes the smallest cube (if it exists) in $\D$ containing both $I$ and $J$. Thus $I\vee J=K$ is equivalent with saying that $I$ and $J$ are subcubes of different children of $K$. This kind of dyadic shifts arise naturally when representing general Calder\'on-Zygmund operators with dyadic shifts as in \cite{H}.  Note that in this case if $A^{\si}_{K}$ is to be non zero then $m,n\geq 1$.

In this situation a weaker form of the quadratic $\mathscr{A}_{p,q}$-condition is sufficient in Theorem \ref{two weight shift}. Namely, let again $Q_{k}, k\in \{1, \dots, 2^{N}\}$, denote the dyadic children of a cube $Q \in \D$. We do not have any special ordering in mind, and in fact the ordering need not be the same for different cubes. Thus, if $Q,Q' \in \D$ and $Q \not=Q'$,  then $Q_{k}$ and $Q'_{k}$ need not be in symmetrical places with respect to the parents $Q$ and $Q'$. We say that the measures $\si$ and $w$ satisfy the quadratic $\mathscr{A}_{p,q}^{*}$-condition if for any $k,l \in \{1,\dots,2^{N}\}, k\not=l$, and any collection $\{a_{Q}\}_{Q \in \D}$ of real numbers  the inequality
\begin{equation}\label{square Ap, disjoint}
\Big\|\Big(\sum_{Q \in \D} \Big(a_{Q}\frac{\si(Q_{k})}{|Q_{k}|}\Big)^{2}1_{Q_{l}}\Big)^{\frac{1}{2}}\Big\|_{L^{q}(w)}
\leq [\si,w]_{p,q}^{*} \Big\|\Big(\sum_{Q \in \D} a_{Q}^{2}1_{Q_{k}}\Big)^{\frac{1}{2}}\Big\|_{L^{p}(\si)}
\end{equation}
is satisfied, and here again  $[\si,w]_{p,q}^{*}$ denotes the best possible constant. Similarly as with the quadratic $\mathscr{A}_{p,q}$-condition above we have $[\si,w]_{p,q}^{*}\simeq [w,\si]_{q',p'}^{*}$. 

The two weight inequality for the Hilbert transform was characterized by M. Lacey, E. Sawyer, C.-Y. Shen and I. Uriarte-Tuero \cite{LSSU} and M. Lacey \cite{L} in the case when the measures $\si$ and $w$ do not have common point masses. This restriction on the measures was lifted by T. Hyt\"onen in \cite{H1}, and a key new component was a similar kind of weakening as we have here of the \emph{Poisson $A_{2}$ conditions}  used in \cite{LSSU} and \cite{L}. 

\begin{theorem}\label{two weight shift, specific form}
Let $1<p,q<\infty$ be two exponents and assume that the measures $\si$ and $w$ satisfy the quadratic $\mathscr{A}^{*}_{p,q}$-condition. Suppose $\mathscr{T}$ is a collection of dyadic shifts with complexities at most $\kappa$, and suppose every shift in $\mathscr{T}$ is of the specific form (\ref{special case}). Then the collection $\mathscr{T}$ is $\mathcal{R}$-bounded from $L^{p}(\si)$ into $L^{q}(w)$ if and only if it satisfies the quadratic testing conditions (\ref{testing direct}) and (\ref{testing dual}), and in this case 
\begin{equation}
\mathcal{R}(\mathscr{T}) \lesssim (1+\kappa)(\mathcal{T}^{\si} + \mathcal{T}^{w}) + (1+\kappa)^{2}[\si,w]^{*}_{p,q}.
\end{equation}

\end{theorem}

We outline the proof Theorem \ref{two weight shift, specific form}. This is very probably known to specialists, but we record this fact here.

All we need to do is to look at the proof above and consider the places where the quadratic $\mathscr{A}_{p,q}$-condition was applied, and show that in this special case it is enough to assume the weaker condition. The quadratic $\mathscr{A}_{p,q}$-condition was applied in two places: first in the end of the subsection  dealing with the case $``Q \cap R = \emptyset$'', and then in the end of the case ``$ Q \subset R \subset Q^{(\kappa)}$''.

Assume that $K \in \D$ and we have an operator $A^{\si}_{K}$ of the form (\ref{special case}). Then for $f \in L^{1}_{loc}(\si)$ and $g \in L^{1}_{loc}(w)$ we have
\begin{equation}\label{average domination in special case}
\begin{split}
\Big|\big\langle A^{\si}_{K}f,g\big\rangle_{w} \Big|
&=\Big|\sum_{\begin{substack}{k,l \in \{1,\dots,2^{N}\} \\ k \not= l }\end{substack}}\sum_{\begin{substack}{I^{(m-1)}=K_{k} \\ J^{(n-1)}= K_{l}}\end{substack}}a_{IJK} \langle f , h_{I}^J\rangle_{\si} \langle g,h_{J}^I\rangle_{w}\Big| \\
& \leq \sum_{\begin{substack}{k,l \in \{1,\dots,2^{N}\} \\ k \not= l  }\end{substack}} \frac{\| 1_{K_{k}}f\|_{L^{1}(\si)}\|1_{K_{l}}g\|_{L^{1}(w)}}{|K|}.
\end{split}
\end{equation}

If we use (\ref{average domination in special case}) in (\ref{first appl. of Ap}) we end up with the term 
\begin{equation*}
\sum_{i,j=1}^{\kappa} \sum_{K \in \D}\sum_{k\not=l} \frac{\|1_{K_{k}}\de^{\si,i}_{K}f_{u}\|_{L^{1}(\si)}\|1_{K_{l}}\de^{w,j}_{K}g_{u}\|_{L^{1}(w)}}{|K|}.
\end{equation*}
If one continues as in (\ref{towards Ap}) with fixed $k\not=l$, the result is  
\begin{equation*}
\begin{split}
&\Big\|\Big(\sum_{u=1}^{U}\sum_{K \in \D} \big(\frac{\|1_{K_{k}}\de^{\si,i}_{K}f_{u}\|_{L^{1}(\si)}}{|K|}\big)^{2}1_{K_{l}}\Big)^{\frac{1}{2}}\Big\|_{L^{q}(w)} \\
&\cdot \Big\|\Big(\sum_{u=1}^{U}\sum_{K \in \D} \big(1_{K_{l}}\de^{w,j}_{K}g_{u}\big)^{2}\Big)^{\frac{1}{2}}\Big\|_{L^{q'}(w)}.
\end{split}
\end{equation*}
The factor related to $g$ is directly handled with Burkholder's inequality, and the other related to $f$ is estimated with the $\mathscr{A}_{p,q}^{*}$-condition similarly as in (\ref{application of Ap}). In the end one can sum over the finite ranges of $k$ and $l$. This takes care of the first application of the $\mathscr{A}_{p,q}^{*}$-condition.

The other application is even easier, since there the functions are already in the right form. If we look at the first term in (\ref{secon appl. of Ap}), we see that it can be written as
\begin{equation*}
\begin{split}
& \sum_{u=1}^{U} \sum_{R \in \D} \sum_{k=1}^{2^{N}} \frac{ \| 1_{\complement R_{k}}\de^{\si,i}_{R}f_{u}  \|_{L^{1}(\si)} 
\|1_{R_{k}} \de^{w}_{R}g_{u} \|_{L^{1}(w)}  }{|R|} \\
&=\sum_{\begin{substack}{k,l: \\ k\not=l}\end{substack}}\sum_{u=1}^{U}\sum_{R \in \D}\frac{ \|1_{R_{l}}\de^{\si,i}_{R}f_{u}  \|_{L^{1}(\si)}  \| 1_{R_{k}}\de^{w}_{R}g_{u} \|_{L^{1}(w)}  }{|R|},
\end{split}
\end{equation*}
and for  a fixed pair $k \not=l$ this can again be estimated with the $\mathscr{A}^{*}_{p,q}$-condition.

\section{Examples related to the quadratic $\mathscr{A}_{p,q}$-condition}

Consider the one weight case with $p=q\in (1,\infty)$, where we have an almost everywhere (in the Lebesgue sense) positive Borel measurable function $w: \R^{N} \to \R$. With the same symbol we also denote the Borel measure 
\begin{equation*}
w(E):= \int_{E} w \ud x,
\end{equation*}
where $E \subset \R^{N}$ is any Borel set. The dual weight to $w$ is $\si := w^{\frac{-1}{p-1}}$, and we again use $\si$ for the corresponding measure. The \emph{Muckenhoupt $A_{p}$ characteristic} is defined as 
\begin{equation*}
[w]_{p}:=\sup_{Q \in \D}\frac{\si(Q)^{p-1}w(Q)}{|Q|^{p}},
\end{equation*}
and the \emph{Muckenhoupt $A_{p}$ class} consists of those weights that have $[w]_{p} < \infty$. 

In this one weight case the weighted Stein's inequality (\ref{two weight Stein}) can be equivalently written as 
\begin{equation}\label{one weight Stein}
\Big\|\Big(\sum_{Q \in \D} \Big(  \frac{\int_{Q}f_{Q} \ud x}{|Q|}\Big)^{2}1_{Q}\Big)^{\frac{1}{2}}\Big\|_{L^{p}(w)}
\leq \mathscr{S} \Big\|\Big(\sum_{Q \in \D} f^{2}_{Q}1_{Q}\Big)^{\frac{1}{2}}\Big\|_{L^{p}(w)}.
\end{equation}

It can quite easily be seen that if $p=2$ the constant $\mathscr{S}$ in the weighted Stein's inequality is $[w]_{2}^{\frac{1}{2}}$, that is the inequality (\ref{one weight Stein}) holds with a finite constant if and only if the weight is in the Muckenhoupt  $A_{2}$ class. A quantitative form of the extrapolation theorem of Rubio de Francia \cite{F} by O. Dragi\v cevi\'c, L. Grafakos, M. Pereyra and S. Petermichl \cite{DGPP} gives then that the best constant $\mathscr{S}(w,p)$ in (\ref{one weight Stein}) satisfies 
\begin{equation*}
\mathscr{S}(w,p) \lesssim \begin{cases} [w]_{p}^{\frac{1}{2(p-1)}}, \ \ 1<p\leq 2, \\ [w]_{p}^{\frac{1}{2}}, \ \ 2\leq p < \infty. \end{cases}
\end{equation*}
Since Lemma \ref{square&Stein} shows that the quadratic $\mathscr{A}_{p,q}$-constant is equivalent to the best constant in the two weight Stein's inequality, we get the quantitative estimates
\begin{equation*}
\begin{cases} [w]_{p}^{\frac{1}{p}} \leq [\si,w]_{p,p} \lesssim  [w]_{p}^{\frac{1}{2(p-1)}}, \ \ &1<p \leq 2, \\ [w]_{p}^{\frac{1}{p}}\leq [\si,w]_{p,p} \lesssim  [w]_{p}^{\frac{1}{2}}, \ \ &2 \leq p < \infty. \end{cases}
\end{equation*}

On the other hand in the general two weight setting the quadratic $\mathscr{A}_{p,q}$-condition is strictly stronger than the simple $A_{p,q}$-condition if $p>2$ or $q<2$: 

\begin{lem}\label{counter ex.}
Let $p,q \in (1,\infty)$ be two exponents.
\begin{itemize}
\item[a)] If $1<p\leq 2\leq q< \infty$, then  $(\si,w)_{p,q}=[\si,w]_{p,q}$ for all Radon measures $\si$ and $w$.
\item[b)] If $2<p<\infty$ or $1<q<2$, then there exist Radon measures $\si$ and $w$ such that $(\si,w)_{p,q}<\infty$ but $[\si,w]_{p,q}=\infty$.
\end{itemize}
\end{lem}

\begin{proof}
The case a) is just Lemma \ref{simple implies square}, so we need to prove only the other assertion. Consider now some exponents $1<p,q<\infty$ and choose a cube $Q_{0} \in \D$ with $|Q_{0}|=1$. Then we simply set the measure $\si$ to be $1_{Q_{0}}\ud x $, that is, the Lebesgue measure restricted to $Q_{0}$.

The measure $w$ that we next construct must satisfy
\begin{equation*}
w(Q) \leq C \frac{|Q|^{q}}{\si(Q)^{\frac{q}{p'}}}, \ \ Q \in \D,
\end{equation*}
for some constant $C$. Keeping this in mind we set $w$ to be
\begin{equation*}
w:= \sum_{k=1}^{\infty}|Q_{0}^{(k)}|^{q-1}1_{Q_{0}^{(k)}\setminus Q_{0}^{(k-1)}}\ud x.
\end{equation*}  

To see that the pair $(\si,w)$ satisfies the simple $A_{p,q}$ condition, first note that since the measures are supported on $Q_{0}$ and $\complement Q_{0}$, respectively,  then $\si(Q)w(Q)=0$ for all cubes $Q \in \D$ with $l(Q) \leq 1$. Also if $Q\in \D$ is such that $l(Q)>1$ and $\si(Q)\not=0$, there exists an $l \in \{1,2,\dots\}$ such that $Q=Q^{(l)}_{0}$. But then
\begin{equation*}
w(Q^{(l)}_{0})=\sum_{k=1}^{l}|Q_{0}^{(k)}|^{q-1}|Q_{0}^{(k)}\setminus Q_{0}^{(k-1)}|
\simeq \sum_{k=1}^{l}|Q_{0}^{(k)}|^{q}
\simeq |Q_{0}^{(l)}|^{q},
\end{equation*}
and this shows that 
\begin{equation*}
\frac{\si(Q^{(l)}_{0})^{\frac{1}{p'}}w(Q^{(l)}_{0})^{\frac{1}{q}}}{|Q^{(l)}_{0}|} \lesssim 1.
\end{equation*}
Thus $(\si,w)_{p,q}\lesssim 1$.

On the other hand consider the quadratic $\mathscr{A}_{p,q}$-condition, and choose some $K \in \{1,2,\dots\}$. We set $a_{k}=1$ for $k \in \{1,\dots, K\}$ and $a_{k}=0$ for $k>K$. Then the construction of the measures shows that
\begin{equation}\label{counter example, left}
\begin{split}
&\Big\|\Big(\sum_{k = 1}^{K} \Big(a_{k}\frac{\si(Q_{0}^{(k)})}{|Q_{0}^{(k)}|}\Big)^{2}1_{Q_{0}^{(k)}}\Big)^{\frac{1}{2}}\Big\|_{L^{q}(w)}^{q} \\
&=\sum_{k=1}^{K} \Big(\sum_{m=k}^{K} |Q_{0}^{(m)}|^{-2}\Big)^{\frac{q}{2}}|Q_{0}^{(k)}|^{q-1} |Q_{0}^{(k)}\setminus Q_{0}^{(k-1)}| \\
&\simeq \sum_{k=1}^{K} |Q_{0}^{(k)}|^{-q+q}=K,
\end{split}
\end{equation}
where in the second to last step we used the fact that a geometric sum is about as big as its biggest term. 

For the quadratic $\mathscr{A}_{p,q}$-condition to hold, this should be dominated by
\begin{equation}\label{counter right computed}
[\si,w]_{p,q}^{q} \Big\|\Big(\sum_{k = 1}^{K}  1_{Q_{0}^{(k)}}\Big)^{\frac{1}{2}}\Big\|_{L^{p}(\si)}^{q}
=[\si,w]_{p,q}^{q} K^{\frac{q}{2}}.
\end{equation}
Comparing (\ref{counter example, left}) and (\ref{counter right computed}), we see that since $K$ was arbitrary, (\ref{counter right computed}) can dominate (\ref{counter example, left}) only if $q \geq 2$. 

So if $q<2$, we can construct a pair $(\si,w)$ of weights such that $(\si,w)_{p,q}<\infty$ but $[\si,w]_{p,q}=\infty$. On the other hand if $p>2$, then $p'<2$, and we can construct measures so that $(\si,w)_{q',p'}= (w,\si)_{p,q}<\infty$ and $[\si,w]_{q',p'}\simeq [w,\si]_{p,q}=\infty$. 

\end{proof}
Combining Lemmas \ref{square&Stein} and \ref{counter ex.} we get the following corollary:
\begin{cor}
If $p,q \in (1,\infty)$ are two exponents, then the simple $A_{p,q}$-condition is sufficient for the two weight Stein's inequality (\ref{two weight Stein}) if and only if $1<p\leq 2\leq q <\infty$.
\end{cor}

\end{document}